\documentclass[10pt,a4paper]{amsart}

\usepackage[T1]{fontenc}
\usepackage{lmodern}
\usepackage{amsmath,amssymb,mathtools}
\usepackage{graphicx}
\usepackage{array}
\usepackage{enumitem}
\usepackage{float}
\usepackage{aliascnt}
\usepackage{needspace}
\usepackage{microtype}
\usepackage{tikz}
\usetikzlibrary{decorations.pathreplacing,arrows.meta}
\usepackage{xcolor}
\usepackage[colorlinks=true,linkcolor=blue,citecolor=blue,urlcolor=blue,
pdftitle={Unique mutation cycles of nucleus-extension quivers},
pdfsubject={Uniqueness of mutation cycles in labeled mutation graphs},
pdfkeywords={quiver mutation, mutation cycle, long mutation cycle, exit, global descent}]{hyperref}
\usepackage[nameinlink,noabbrev,capitalise]{cleveref}

\newtheorem{theorem}{Theorem}[section]

\newaliascnt{proposition}{theorem}
\newtheorem{proposition}[proposition]{Proposition}
\aliascntresetthe{proposition}

\newaliascnt{lemma}{theorem}
\newtheorem{lemma}[lemma]{Lemma}
\aliascntresetthe{lemma}

\newaliascnt{corollary}{theorem}
\newtheorem{corollary}[corollary]{Corollary}
\aliascntresetthe{corollary}

\theoremstyle{definition}
\newaliascnt{definition}{theorem}
\newtheorem{definition}[definition]{Definition}
\aliascntresetthe{definition}

\newaliascnt{example}{theorem}
\newtheorem{example}[example]{Example}
\aliascntresetthe{example}

\theoremstyle{definition}
\newaliascnt{remark}{theorem}
\newtheorem{remark}[remark]{Remark}
\aliascntresetthe{remark}

\AddToHook{env/theorem/begin}{\crefalias{section}{theorem}}
\AddToHook{env/proposition/begin}{\crefalias{section}{proposition}}
\AddToHook{env/lemma/begin}{\crefalias{section}{lemma}}
\AddToHook{env/corollary/begin}{\crefalias{section}{corollary}}
\AddToHook{env/definition/begin}{\crefalias{section}{definition}}
\AddToHook{env/example/begin}{\crefalias{section}{example}}
\AddToHook{env/remark/begin}{\crefalias{section}{remark}}

\numberwithin{equation}{section}
\numberwithin{table}{section}

\crefname{equation}{equation}{equations}
\Crefname{equation}{Equation}{Equations}
\crefname{table}{Table}{Tables}
\Crefname{table}{Table}{Tables}
\crefname{figure}{Figure}{Figures}
\Crefname{figure}{Figure}{Figures}

\newcommand{\nminus}{N^{-}}
\newcommand{\nplus}{N^{+}}
\newcommand{\Rclass}{\mathcal{R}}
\newcommand{\SSclass}{\mathcal{SS}}
\newcommand{\Sclass}{\mathcal{S}}
\newcommand{\Qalt}[1]{Q^{#1}}
\newenvironment{weightcalc}
  {\par\addvspace{4pt}\indent$\displaystyle\begin{aligned}}
  {\end{aligned}$\par\addvspace{4pt}\ignorespacesafterend}

\newcommand{\Xj}[1]{X^{#1}}
\newcommand{\op}{\mathrm{op}}

\tikzset{fivearrow/.style={->,>=stealth,line width=.8pt,shorten >=5pt,shorten <=5pt},
         triarr/.style={->,>=stealth,line width=.8pt,shorten >=4pt,shorten <=4pt}}
\newcommand{\fivecoords}{%
  \coordinate (cm) at (-1.85,0.00);
  \coordinate (e)  at (-0.67,1.18);
  \coordinate (w)  at ( 0.67,1.18);
  \coordinate (cp) at ( 1.85,0.00);
  \coordinate (d)  at (0,-1.05);
}
\newcommand{\fivelabels}{%
  \node[fill=none,inner sep=1pt,font=\small] at (cm) {$N^-$};
  \node[fill=none,inner sep=1pt,font=\small] at (e) {$e$};
  \node[fill=none,inner sep=1pt,font=\small] at (w) {$w$};
  \node[fill=none,inner sep=1pt,font=\small] at (cp) {$N^+$};
  \node[fill=none,inner sep=1pt,font=\small] at (d) {$D$};
}

\newcommand{\GraphOne}{%
\begin{tikzpicture}[scale=.91]
\fivecoords
\draw[fivearrow] (w)--(e); \draw[fivearrow] (e)--(d); \draw[fivearrow] (d)--(w);
\draw[fivearrow] (e)--(cm); \draw[fivearrow] (cm)--(w); \draw[fivearrow] (cm)--(d);
\draw[fivearrow] (cp)--(e); \draw[fivearrow] (w)--(cp); \draw[fivearrow,shorten >=9pt] (d)--(cp);
\draw[fivearrow,shorten >=10pt] (cm) to[bend right=14] (cp);
\fivelabels
\end{tikzpicture}}
\newcommand{\GraphTwo}{%
\begin{tikzpicture}[scale=.91]
\fivecoords
\draw[fivearrow] (d)--(e); \draw[fivearrow] (e)--(w); \draw[fivearrow] (w)--(d);
\draw[fivearrow] (cm)--(e); \draw[fivearrow] (w)--(cm); \draw[fivearrow] (cm)--(d);
\draw[fivearrow] (cp)--(e); \draw[fivearrow] (cp)--(w); \draw[fivearrow,shorten >=9pt] (d)--(cp);
\draw[fivearrow,shorten >=10pt] (cm) to[bend right=14] (cp);
\fivelabels
\end{tikzpicture}}
\newcommand{\GraphTwoAtOne}{%
\begin{tikzpicture}[scale=.91]
\fivecoords
\draw[fivearrow] (e)--(d); \draw[fivearrow] (e)--(w); \draw[fivearrow] (w)--(d);
\draw[fivearrow] (cm)--(e); \draw[fivearrow] (w)--(cm); \draw[fivearrow] (cm)--(d);
\draw[fivearrow] (cp)--(e); \draw[fivearrow] (cp)--(w); \draw[fivearrow,shorten >=9pt] (d)--(cp);
\draw[fivearrow,shorten >=10pt] (cm) to[bend right=14] (cp);
\fivelabels
\end{tikzpicture}}
\newcommand{\GraphOdd}{%
\begin{tikzpicture}[scale=.91]
\fivecoords
\draw[fivearrow] (w)--(e); \draw[fivearrow] (e)--(d); \draw[fivearrow] (d)--(w);
\draw[fivearrow] (e)--(cm); \draw[fivearrow] (cm)--(w); \draw[fivearrow] (cm)--(d);
\draw[fivearrow] (e)--(cp); \draw[fivearrow] (cp)--(w); \draw[fivearrow,shorten >=9pt] (d)--(cp);
\draw[fivearrow,shorten >=10pt] (cm) to[bend right=14] (cp);
\fivelabels
\end{tikzpicture}}
\newcommand{\GraphEven}{%
\begin{tikzpicture}[scale=.91]
\fivecoords
\draw[fivearrow] (d)--(e); \draw[fivearrow] (e)--(w); \draw[fivearrow] (w)--(d);
\draw[fivearrow] (cm)--(e); \draw[fivearrow] (w)--(cm); \draw[fivearrow] (cm)--(d);
\draw[fivearrow] (cp)--(e); \draw[fivearrow] (w)--(cp); \draw[fivearrow,shorten >=9pt] (d)--(cp);
\draw[fivearrow,shorten >=10pt] (cm) to[bend right=14] (cp);
\fivelabels
\end{tikzpicture}}
\newcommand{\GraphTwoK}{%
\begin{tikzpicture}[scale=.91]
\fivecoords
\draw[fivearrow] (e)--(w); \draw[fivearrow] (w)--(d); \draw[fivearrow] (e)--(d);
\draw[fivearrow] (cm)--(e); \draw[fivearrow] (w)--(cm); \draw[fivearrow] (cm)--(d);
\draw[fivearrow] (cp)--(e); \draw[fivearrow] (w)--(cp); \draw[fivearrow,shorten >=9pt] (d)--(cp);
\draw[fivearrow,shorten >=10pt] (cm) to[bend right=14] (cp);
\fivelabels
\end{tikzpicture}}
\newcommand{\GraphTwoKOne}{%
\begin{tikzpicture}[scale=.91]
\fivecoords
\draw[fivearrow] (w)--(d); \draw[fivearrow] (d)--(e); \draw[fivearrow] (w)--(e);
\draw[fivearrow] (e)--(cm); \draw[fivearrow] (cm)--(w); \draw[fivearrow] (cm)--(d);
\draw[fivearrow] (e)--(cp); \draw[fivearrow] (cp)--(w); \draw[fivearrow,shorten >=9pt] (d)--(cp);
\draw[fivearrow,shorten >=10pt] (cm) to[bend right=14] (cp);
\fivelabels
\end{tikzpicture}}
\newcommand{\GraphTwoKTwo}{%
\begin{tikzpicture}[scale=.91]
\fivecoords
\draw[fivearrow] (d)--(e); \draw[fivearrow] (e)--(w); \draw[fivearrow] (d)--(w);
\draw[fivearrow] (cm)--(e); \draw[fivearrow] (w)--(cm); \draw[fivearrow] (cm)--(d);
\draw[fivearrow] (cp)--(e); \draw[fivearrow] (w)--(cp); \draw[fivearrow,shorten >=9pt] (d)--(cp);
\draw[fivearrow,shorten >=10pt] (cm) to[bend right=14] (cp);
\fivelabels
\end{tikzpicture}}
\title[Unique mutation cycles]{Unique mutation cycles of nucleus-extension quivers}

\author{Yuhang Li}
\address{School of Mathematical Sciences, Zhejiang University of Technology, Hangzhou 310023, China}
\email{color3437009@gmail.com}

\author{Haiyan Zhu$^{*}$}
\address{School of Mathematical Sciences, Zhejiang University of Technology, Hangzhou 310023, China}
\thanks{$^{*}$Corresponding author.}
\thanks{Supported by the National Natural Science Foundation of China (12271481).}
\email{hyzhu@zjut.edu.cn}

\begin{document}
\begin{abstract}
We prove the uniqueness conjecture of Fomin and Neville for their long
mutation cycles in all ranks $n\ge4$. More generally, we show that every
nucleus-extension quiver has a unique simple cycle in its labeled mutation
graph. These quivers allow an acyclic part of arbitrary size outside the
nucleus. Our proof uses three mutation statuses determined by a closed
nucleus and the full subquivers obtained by deleting one part of its nucleus
decomposition at a time from the ambient quiver. We analyze transitions
between these statuses to prove that every mutation leaving the prescribed
cycle is performed at an exit.
\end{abstract}

\maketitle
\enlargethispage{2pt}

\section{Introduction}\label{sec:introduction}

\emph{Quiver mutation}, introduced by Fomin and Zelevinsky in the theory of
cluster algebras \cite{FZ02}, is an elementary operation on quivers. The
labeled quivers in a \emph{mutation class} and the single mutations between
them form its \emph{labeled mutation graph}. A \emph{mutation cycle} is a
reduced closed walk in this graph, and it is \emph{simple} if no quiver
is repeated except at the beginning and end. Examples arise from
triangulated surfaces \cite{FST08}, periodicity for pairs of Dynkin diagrams
\cite{Kel13}, flip moves in plabic graphs \cite{BW20}, and mutation-periodic
quivers \cite{FM11}.

Cao and Li \cite{CL19} proved that triangular extensions preserve the
existence of maximal green sequences and observed that the same argument
applies to reddening sequences. Fomin and Neville \cite{FN25} constructed
simple mutation cycles of length $n+4k$ for every $n\ge4$ and $k\ge1$, so
their lengths are unbounded even for a fixed number of vertices. They
conjectured that each cycle in this family is the unique simple cycle in its
labeled mutation graph and proved the conjecture for $n=4$. They also noted
that uniqueness would make the associated cluster modular groups
\cite{FG09} infinite cyclic. Using reddening sequences and triangular
extensions, Ervin and Neville \cite[Theorem~4.24]{EN26} extended the
Fomin--Neville construction to a broader class of simple mutation cycles.

In this paper, we prove the uniqueness conjecture posed by Fomin and Neville
in \emph{Long mutation cycles} \cite{FN25}. More generally, \cref{thm:main}
establishes uniqueness for \emph{nucleus-extension quivers}, which form a subfamily
of the Ervin--Neville construction and allow an acyclic part of arbitrary
size outside the nucleus. We use results on forks from Warkentin \cite{War14} and
Ervin \cite{Erv24}, while following the terminology and notation of Fomin and Neville
\cite{FN25}. The key idea is to study a quiver $Q$ relative to a \emph{closed nucleus}
$Q|_N$.
The global descent $e$ of the nucleus determines a decomposition
$N=D\sqcup\{e\}\sqcup U$ of its vertex set, where the vertices in $D$ point to
$e$, while $e$ points to those in $U$. Removing $\{e\}$, $D$, and $U$ separately
from $V(Q)$ gives three vertex sets inducing the full subquivers
$Q|_{\{e\}^C}$, $Q|_{D^C}$, and $Q|_{U^C}$. This configuration is the basis for
three \emph{mutation statuses}, $\Rclass$, $\Sclass$, and $\SSclass$. We show how
the structure defining the $\Rclass$-status reappears under mutation, possibly
with a different closed nucleus. This allows us to analyze reduced mutation sequences
inductively even when the exit criterion based on strict growth of the sum
of edge weights does not apply to the initial mutation, and to prove that
every mutation leaving the prescribed cycle is performed at an exit.

\Needspace{3\baselineskip}
\smallskip
\noindent\textbf{Organization.}
\Cref{sec:preliminaries} recalls the basic notions of quiver mutation, mutation cycles,
exits, and the rank-three results used later. \Cref{sec:classes} introduces nucleus,
closed nucleus, and the mutation statuses $\Rclass$, $\Sclass$, and $\SSclass$, and
establishes their basic structural properties. \Cref{sec:mutation-properties} studies
how these statuses behave under mutation and derives the exit results needed for the
proof. \Cref{sec:cycle} applies these results to nucleus-extension quivers and proves
the uniqueness of the prescribed simple cycle.

\section{Preliminaries}\label{sec:preliminaries}

We recall the basic notions of quivers and quiver mutation following
\cite{FWZ21}.

For integers $a$ and $b$, write $[a,b]=\{a,a+1,\ldots,b\}$, interpreted
as empty when $a>b$, and abbreviate $[1,n]$ to $[n]$.

A \emph{quiver} is a directed graph without loops or oriented $2$-cycles;
parallel arrows are allowed. All quivers are finite and labeled, and equality
means equality as labeled directed graphs. We write $V(Q)$ for the vertex set of $Q$.

For $S\subseteq V(Q)$, the \emph{full subquiver} $Q|_S$ has vertex set $S$
and all arrows of $Q$ joining vertices of $S$. Throughout the paper, a
subquiver means a full subquiver. We write $S^C:=V(Q)\setminus S$.

The skew-symmetric \emph{exchange matrix} $B(Q)=(b_{ij}(Q))$ is defined by
\[
 b_{ij}(Q)=\#\{\text{arrows }i\to j\}
           -\#\{\text{arrows }j\to i\}.
\]
Thus $b_{ij}>0$ precisely when $i\to j$, and $|b_{ij}|$ is the \emph{weight}
between $i$ and $j$. Following Warkentin \cite{War14}, a quiver is
\emph{abundant} if $|b_{ij}|\geq2$ for all distinct vertices $i,j$.

For disjoint vertex sets $A,B\subseteq V(Q)$, write $A\to B$ if $a\to b$
for every $a\in A$ and $b\in B$; the same notation is used when either set
is a singleton. We denote by $Q^{\op}$ the \emph{global reversal} of $Q$,
obtained by reversing all arrows.

A vertex is a \emph{source}, respectively a \emph{sink}, if every incident
arrow points away from it, respectively toward it. A quiver is \emph{acyclic}
if it has no oriented cycles. A nonempty abundant acyclic quiver has a
unique source and a unique sink.

\begin{definition}[{\cite{FZ02}}]\label{def:quiver-mutation}
Let $Q$ be a quiver with exchange matrix $B(Q)=(b_{ij})$ and let
$v\in V(Q)$. The \emph{mutation} of $Q$ at $v$, denoted by $\mu_v(Q)$,
is the quiver with exchange matrix $B(\mu_v(Q))=(b'_{ij})$, where
\[
 b'_{ij}=
 \begin{cases}
  -b_{ij},& i=v \text{ or } j=v,\\[1mm]
  b_{ij}+[b_{iv}]_+[b_{vj}]_+-[-b_{iv}]_+[-b_{vj}]_+,
    & i,j\neq v,
 \end{cases}
 \qquad [x]_+=\max\{x,0\}.
\]
\end{definition}

Mutation is an involution and commutes with global reversal. If $v\in S$, then
$\mu_v(Q)|_S=\mu_v(Q|_S)$. Mutation at a source or a sink only reverses the
arrows incident to that vertex.

A \emph{mutation sequence} is a finite sequence
$\mathbf i=(i_1,\ldots,i_s)$ of vertices of $Q$. Set
\[
 \mu_{\mathbf i}=\mu_{i_s}\circ\cdots\circ\mu_{i_1}.
\]
Thus the mutations in a displayed sequence are performed from left to right.
The sequence $\mathbf i$ is \emph{reduced} if consecutive vertices are
distinct. For distinct vertices $i,j$ and $k\geq0$, write
$(i,j)^k=(i,j,i,j,\ldots,i,j)$ for the sequence with $2k$ terms, with
$(i,j)^0$ the empty sequence.

We recall the following definitions and lemmas from \cite{FN25}.

Two labeled quivers are \emph{mutation-equivalent} if one can be obtained
from the other by a mutation sequence. The \emph{mutation class} of $Q$
consists of all labeled quivers mutation-equivalent to $Q$. Its \emph{labeled
mutation graph} has these quivers as vertices and an edge labeled $i$ between
$X$ and $\mu_i(X)$.

\begin{definition}\label{def:mutation-cycle}
Let $X$ be a quiver and let $\mathbf i=(i_1,\ldots,i_\ell)$ be a mutation
sequence with $\ell\geq1$. Set
\[
 X^0:=X,\qquad
 X^j:=\mu_{(i_1,\ldots,i_j)}(X)\quad(1\leq j\leq \ell).
\]
Suppose that
\[
 X^\ell=X^0,\qquad X^{j-1}\neq X^j\quad(1\leq j\leq \ell),
\]
and that $i_j\neq i_{j+1}$ for $1\leq j<\ell$ and $i_\ell\neq i_1$. Then the closed walk
\[
 X^0\mathrel{\mathop{\rule{1.5em}{0.4pt}}\limits^{i_1}}X^1
 \mathrel{\mathop{\rule{1.5em}{0.4pt}}\limits^{i_2}}\cdots
 \mathrel{\mathop{\rule{1.5em}{0.4pt}}\limits^{i_\ell}}X^\ell=X^0
\]
is a \emph{mutation cycle}. The cycle is
\emph{simple} if $X^0,\ldots,X^{\ell-1}$ are pairwise distinct.
\end{definition}

\begin{definition}\label{def:exit}
A vertex $i$ in a quiver $Q$ is an \emph{exit} if
$\mu_{\mathbf j}(Q)\neq Q$ for every reduced mutation sequence
$\mathbf j=(j_1,\ldots,j_\ell)$, $\ell\geq1$, with $j_1=i$.
Consequently, the edge labeled $i$ between $Q$ and $\mu_i(Q)$ in the labeled
mutation graph does not lie on any mutation cycle.
\end{definition}

We next recall the required facts about $3$-vertex quivers. The term
\emph{cyclic} will always refer to a $3$-vertex subquiver containing an
oriented $3$-cycle. An acyclic $3$-vertex quiver with nonzero weights has one
source and one sink; the remaining vertex is its \emph{elbow}.

A vertex of a $3$-vertex quiver is an \emph{ascent}, respectively a
\emph{descent}, if mutation at that vertex strictly increases, respectively
decreases, the weight of the opposite edge.

\begin{lemma}\label{lem:rank-three-descent}
Let $Q$ be an abundant $3$-vertex quiver.
\begin{enumerate}[label=\textup{(\roman*)},leftmargin=2.2em,itemsep=2pt]
\item If $Q$ is acyclic, then it has exactly one ascent and no descents.
\item If $Q$ is cyclic, then it has at most one descent. If a descent exists,
then it is unique, lies opposite the unique largest weight, and the other two
vertices are ascents.
\end{enumerate}
\end{lemma}

\begin{lemma}\label{lem:rank-three-edge-invariance}
Let $i,j,u,v$ be four distinct vertices of a quiver $Q$. If $i$ is an
ascent or a descent in each of $Q|_{\{i,j,u\}}$ and
$Q|_{\{i,j,v\}}$, then
\[
 b_{uv}(Q)=b_{uv}(\mu_i(Q)).
\]
\end{lemma}

For $a\ge2$, define $p_\ell(a)$ for $\ell\ge-1$ by the Chebyshev recurrence
\begin{equation}\label{eq:rank-three-recurrence}
 p_{-1}(a)=0,\qquad p_0(a)=1,\qquad
 p_{\ell+1}(a)=a p_\ell(a)-p_{\ell-1}(a)\quad(\ell\ge0).
\end{equation}
\begin{lemma}\label{lem:rank-three-alternating}
Let $Q$ be an abundant acyclic quiver on the vertex set
$\{1,2,3\}$, whose elbow is $1$. Without loss of generality, write
\[
 3\xrightarrow{\,b\,}1\xrightarrow{\,a\,}2,
 \qquad 3\xrightarrow{\,c\,}2,
\]
where $a,b,c\ge2$, and put $p_\ell=p_\ell(a)$. For $s\ge0$, set
\[
 Q^0=Q,\qquad Q^{2s}=\mu_{(1,2)^s}(Q),\qquad
 Q^{2s+1}=\mu_1\bigl(Q^{2s}\bigr).
\]
For $s\ge0$,
\[
 \begin{aligned}
 b_{21}(Q^{2s+1})&=a,\\
 b_{13}(Q^{2s+1})&=p_{2s}b+p_{2s-1}c,\\
 b_{32}(Q^{2s+1})&=p_{2s+1}b+p_{2s}c,
 \end{aligned}
\]
and, for $s\ge1$,
\[
 \begin{aligned}
 b_{12}(Q^{2s})&=a,\\
 b_{23}(Q^{2s})&=p_{2s-1}b+p_{2s-2}c,\\
 b_{31}(Q^{2s})&=p_{2s}b+p_{2s-1}c.
 \end{aligned}
\]
In particular, $Q^{j}$ is cyclic for $j\ge1$.
\end{lemma}

\Needspace{10\baselineskip}
\section{Nucleus-based mutation statuses}\label{sec:classes}

We introduce three nucleus-based mutation statuses and establish their
structural properties.

\subsection{Nucleus and closed nucleus}\label{subsec:the-nucleus}

Before defining nucleus, we recall the notions of vortex and global descent
from \cite[Definitions~6.2 and 6.3]{FN25}.

\begin{definition}\label{def:global-descent}
A quiver $Q$ has a \emph{global descent} at a vertex $g$ if
$Q$ contains at least one cyclic subquiver and every cyclic subquiver of $Q$
has descent $g$.
\end{definition}

\begin{remark}\label{rem:delete-global-descent}
Let $Q$ be an abundant quiver with a global descent at $e$. Then
$Q|_{\{e\}^C}$ is acyclic; see \cite[Lemma~6.5]{FN25}.
\end{remark}

\begin{definition}\label{def:vortex}
A \emph{vortex} is a $4$-vertex quiver such that all its weights are
nonzero, one vertex is a source or a sink, and the other three vertices
form a cyclic subquiver. The unique source or sink is called the \emph{apex}
of the vortex. A quiver is \emph{vortex-free} if none of its full $4$-vertex
subquivers is a vortex.
\end{definition}

\begin{definition}\label{def:nucleus}
A full subquiver $Q|_N$ is a \emph{nucleus} of $Q$ if it is abundant,
vortex-free, and has a global descent at $e$. Its \emph{nucleus decomposition} is
$N=D\sqcup\{e\}\sqcup U$, where
\[
D:=\{v\in N:v\longrightarrow e\},\qquad U:=\{v\in N:e\longrightarrow v\}.
\]
Moreover, the nucleus $Q|_N$ is called a \emph{closed nucleus} of $Q$ if
\[
 N=\{v\in V(Q):v\text{ belongs to a cyclic subquiver of }Q
 \text{ with descent }e\}.
\]
\end{definition}

The basic lemmas in this subsection are reformulations or consequences of
results in \cite[Section~6]{FN25}.

\begin{lemma}\label{lem:global-descent-structure}
Let $Q|_N$ be a nucleus with decomposition $N=D\sqcup\{e\}\sqcup U$. Then $Q|_{\{d,e,u\}}$ is cyclic with descent $e$
for every $d\in D$ and $u\in U$.
\end{lemma}

\begin{proof}
Since $Q|_N$ is a nucleus, there exist $d_1\in D$ and $u_1\in U$ such that
$Q|_{\{e,d_1,u_1\}}$ is cyclic with descent $e$.

It remains to prove that $U\to D$. We first claim that $u_1\to D$ for the fixed
vertex $u_1$ above. Let $d\in D$. If $d=d_1$, this follows from the choice of
$d_1$ and $u_1$. Otherwise, suppose that $d\to u_1$. By abundance, either $d_1\to d$ or
$d\to d_1$. If $d_1\to d$, then $Q|_{\{u_1,d_1,d\}}$ is cyclic and does not
contain $e$, which is impossible since $Q|_N$ has global descent at $e$. If
$d\to d_1$, then $Q|_{\{d_1,e,u_1,d\}}$ is a vortex with apex $d$, while
$Q|_N$ is vortex-free; see \cref{fig:global-descent-contradictions}. Hence
$u_1\to d$.

\begin{figure}[H]
\centering
\begin{tikzpicture}[>=stealth,scale=.95,
  every node/.style={font=\small,inner sep=1pt},
  weight/.style={fill=none,inner sep=1.1pt,font=\scriptsize}]
  \coordinate (d1) at (-1.35,0.95);
  \coordinate (e) at (0,1.8);
  \coordinate (u1) at (1.35,0.95);
  \coordinate (d) at (0,-0.35);

  \draw[->,shorten >=5pt,shorten <=5pt] (d1)--(e);
  \draw[->,shorten >=5pt,shorten <=5pt] (e)--(u1);
  \draw[->,shorten >=5pt,shorten <=5pt] (u1)--(d1);
  \draw[->,red,very thick,dashed,shorten >=5pt,shorten <=5pt] (d)--(d1);
  \draw[->,shorten >=5pt,shorten <=5pt] (d)--(e);
  \draw[->,red,very thick,dashed,shorten >=5pt,shorten <=5pt] (d)--(u1);

  \node[fill=none] at (d1) {$d_1$};
  \node[fill=none] at (e) {$e$};
  \node[fill=none] at (u1) {$u_1$};
  \node[fill=none] at (d) {$d$};
\end{tikzpicture}
\caption{}
\label{fig:global-descent-contradictions}
\end{figure}

Then we claim that $u\to d$ for any $d\in D$ and $u\in U$.
Assume that $d\to u$.
Since $u_1\to d$, we have $u\ne u_1$. If $u\to u_1$, then $Q|_{\{u_1,d,u\}}$
is cyclic and does not contain $e$; otherwise, $Q|_{\{d,e,u_1,u\}}$ is a
vortex with apex $u$. Hence $U\to D$.
\end{proof}

\begin{remark}\label{rem:nucleus-no-source-sink}
A nucleus has neither a source nor a sink.
\end{remark}

\Needspace{7\baselineskip}
For a quiver $Q$ and $i\in V(Q)$, we refer to conditions (6.3)--(6.5) of
\cite{FN25} as F1--F3, respectively:
\begin{enumerate}[label=(F\arabic*),leftmargin=3.2em,itemsep=1pt,topsep=2pt]
\item $i$ is an ascent in every cyclic subquiver of $Q$ containing $i$;
\item $i$ is neither a source nor a sink in $Q$;
\item $i$ is not the apex of a vortex in $Q$.
\end{enumerate}

\begin{lemma}\label{lem:basic-f1-f3}
Let $R$ be either an abundant acyclic quiver or a nucleus, and let
$v\in V(R)$. If $R$ is acyclic, assume that $v$ is neither a source nor a
sink; if $R$ is a nucleus, assume that $v$ is not its global descent. Then
$(R,v)$ satisfies F1--F3.
\end{lemma}

\begin{lemma}\label{lem:mutation-equivalent-nucleus}
Let $Q$ be an abundant quiver, and let $(i_1,\ldots,i_m)$ be a reduced
mutation sequence with $m\geq1$. If $(Q,i_1)$ satisfies F1--F3, then
$\mu_{(i_1,\ldots,i_m)}(Q)$ is a nucleus with global descent at $i_m$.
\end{lemma}

\begin{lemma}\label{lem:f1-f3-exit}
Let $Q$ be an abundant quiver and let $v\in V(Q)$. If $(Q,v)$ satisfies
F1--F3, then $v$ is an exit in $Q$.
\end{lemma}

In Sections~\ref{sec:mutation-properties} and~\ref{sec:cycle}, we mainly use
the following two special types of nucleus.

\begin{definition}\label{def:special-nucleus-types}
Let $Q|_N$ be a full subquiver, and let $N=D\sqcup\{e\}\sqcup U$
with $D,U\ne\varnothing$.
\begin{enumerate}[label=(\roman*),leftmargin=2.2em,itemsep=4pt]
\item We call $Q|_N$ a \emph{$1$-nucleus} if $\mu_e(Q|_N)$ is abundant
and acyclic with orientation
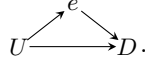
\[
\begin{tikzpicture}[>=stealth,scale=.68,baseline=-.5ex,
  every node/.style={font=\small,inner sep=1pt}]
  \coordinate (u) at (-1.05,0);
  \coordinate (e) at (0,0.82);
  \coordinate (d) at (1.05,0);
  \draw[->,shorten >=4pt,shorten <=4pt] (u)--(e);
  \draw[->,shorten >=4pt,shorten <=4pt] (e)--(d);
  \draw[->,shorten >=4pt,shorten <=4pt] (u)--(d);
  \node[fill=none] at (u) {$U$};
  \node[fill=none] at (e) {$e$};
  \node[fill=none] at (d) {$D$};
\end{tikzpicture}.
\]
\item For an integer $k\geq1$, we call $Q|_N$ a \emph{$2k$-nucleus} if
there is a vertex $w\in N\setminus\{e\}$ such that $U=\{w\}$ or $D=\{w\}$,
and $\mu_{(e,w)^k}(Q|_N)$ is abundant and acyclic with the corresponding orientation
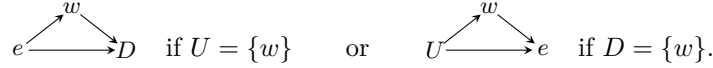
\[
\begin{tikzpicture}[>=stealth,scale=.68,baseline=-.5ex,
  every node/.style={font=\small,inner sep=1pt}]
  \coordinate (e) at (-1.05,0);
  \coordinate (w) at (0,0.82);
  \coordinate (d) at (1.05,0);
  \draw[->,shorten >=4pt,shorten <=4pt] (e)--(w);
  \draw[->,shorten >=4pt,shorten <=4pt] (w)--(d);
  \draw[->,shorten >=4pt,shorten <=4pt] (e)--(d);
  \node[fill=none] at (e) {$e$};
  \node[fill=none] at (w) {$w$};
  \node[fill=none] at (d) {$D$};
\end{tikzpicture}\quad\text{if }U=\{w\}
\qquad\text{or}\qquad
\begin{tikzpicture}[>=stealth,scale=.68,baseline=-.5ex,
  every node/.style={font=\small,inner sep=1pt}]
  \coordinate (u) at (-1.05,0);
  \coordinate (w) at (0,0.82);
  \coordinate (e) at (1.05,0);
  \draw[->,shorten >=4pt,shorten <=4pt] (u)--(w);
  \draw[->,shorten >=4pt,shorten <=4pt] (w)--(e);
  \draw[->,shorten >=4pt,shorten <=4pt] (u)--(e);
  \node[fill=none] at (u) {$U$};
  \node[fill=none] at (w) {$w$};
  \node[fill=none] at (e) {$e$};
\end{tikzpicture}\quad\text{if }D=\{w\}.
\]
The vertex $w$ is called the \emph{alternating vertex}.
\end{enumerate}
\end{definition}

By \cref{lem:basic-f1-f3,lem:mutation-equivalent-nucleus}, in both cases
$Q|_N$ is a nucleus with global descent at $e$ and nucleus decomposition
$N=D\sqcup\{e\}\sqcup U$.

\Needspace{12\baselineskip}
In general, an abundant quiver need not have a unique closed nucleus.

\begin{example}\label{ex:two-based-nucleus}
Let $Q$ be the abundant quiver shown below. Black arrows have weight $2$,
and blue arrows have weight $6$.
\begin{center}
\begin{tikzpicture}[>=stealth,scale=.93,
  every node/.style={font=\small,inner sep=1.2pt},draw=black]
  \coordinate (v1) at (-2.05,1.25);
  \coordinate (v2) at (-3.00,0);
  \coordinate (v3) at (-2.05,-1.25);
  \coordinate (v4) at (2.05,1.25);
  \coordinate (v5) at (3.00,0);
  \coordinate (v6) at (2.05,-1.25);

  \draw[->,solid,shorten >=5pt,shorten <=5pt]
    (v1)--(v4);
  \draw[->,solid,shorten >=5pt,shorten <=5pt]
    (v1)--(v5);
  \draw[->,solid,shorten >=5pt,shorten <=5pt]
    (v1)--(v6);
  \draw[->,solid,shorten >=5pt,shorten <=5pt]
    (v2)--(v4);
  \draw[->,solid,shorten >=5pt,shorten <=5pt]
    (v2)--(v5);
  \draw[->,solid,shorten >=5pt,shorten <=5pt]
    (v2)--(v6);
  \draw[->,solid,shorten >=5pt,shorten <=5pt]
    (v3)--(v4);
  \draw[->,solid,shorten >=5pt,shorten <=5pt]
    (v3)--(v5);
  \draw[->,solid,shorten >=5pt,shorten <=5pt]
    (v3)--(v6);

  \draw[->,shorten >=5pt,shorten <=5pt] (v1)--(v2);
  \draw[->,shorten >=5pt,shorten <=5pt] (v2)--(v3);
  \draw[->,blue,shorten >=5pt,shorten <=5pt] (v3)--(v1);
  \draw[->,shorten >=5pt,shorten <=5pt] (v4)--(v5);
  \draw[->,shorten >=5pt,shorten <=5pt] (v5)--(v6);
  \draw[->,blue,shorten >=5pt,shorten <=5pt] (v6)--(v4);

  \node[fill=none] at (v1) {$1$};
  \node[fill=none] at (v2) {$2$};
  \node[fill=none] at (v3) {$3$};
  \node[fill=none] at (v4) {$4$};
  \node[fill=none] at (v5) {$5$};
  \node[fill=none] at (v6) {$6$};
\end{tikzpicture}
\end{center}
Each of the full subquivers $Q|_{\{1,2,3\}}$ and $Q|_{\{4,5,6\}}$ is a
closed $1$-nucleus of $Q$.
\end{example}

\Needspace{13\baselineskip}
\subsection{Mutation statuses}\label{subsec:status-unique-nucleus}

We now introduce three mutation statuses and prove that any quiver with one
of these statuses has a unique closed nucleus.

\begin{definition}\label{def:classes}
Let $Q$ be an abundant quiver with a closed nucleus decomposition
$N=D\sqcup\{e\}\sqcup U$. We define three mutation statuses as follows.
\begin{enumerate}[label=(\roman*),leftmargin=2.2em,itemsep=2pt]
\item \emph{$\Rclass$-status}.
We write $Q\in\Rclass$ and call $r$ the \emph{return} of $Q$ if
$Q|_{\{e\}^C}$, $Q|_{D^C}$, and $Q|_{U^C}$ are each a nucleus of $Q$ with a
common global descent at $r$.
\item \emph{$\Sclass$-status}.
We write $Q\in\Sclass$ if $Q|_{\{e\}^C}$, $Q|_{D^C}$, and $Q|_{U^C}$
are acyclic and $Q$ has either a source but no sink, or a sink but no source.
\item \emph{$\SSclass$-status}.
We write $Q\in\SSclass$ if $Q|_{\{e\}^C}$, $Q|_{D^C}$, and $Q|_{U^C}$
are acyclic and $Q$ has a source and a sink.
\end{enumerate}
\end{definition}

If the closed nucleus is a $1$-nucleus, we write $\Rclass_1$, $\Sclass_1$,
and $\SSclass_1$; if it is a $2k$-nucleus, we write $\Rclass_{2k}$,
$\Sclass_{2k}$, and $\SSclass_{2k}$.

\begin{remark}\label{rem:global-reversal}
Global reversal preserves the mutation status and nucleus type. It
interchanges $D$ and $U$, swaps sources and sinks, and leaves the return in
the $\Rclass$-case unchanged.
\end{remark}

\begin{example}\label{ex:mutation-statuses}
Let $N=\{1,2,3\}$, $e=2$, $D=\{1\}$, and $U=\{3\}$. Blue arrows have weight $6$, and black arrows have weight $2$. Then
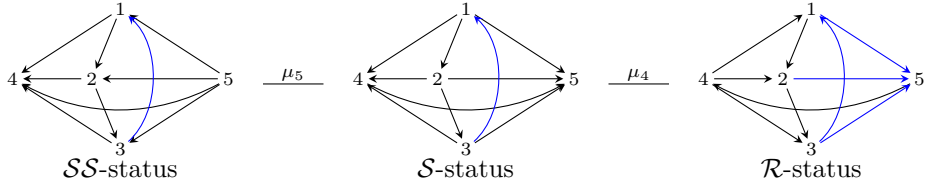
\begin{center}
\setlength{\tabcolsep}{2pt}
\begin{tabular}{@{}c@{\quad}c@{\quad}c@{\quad}c@{\quad}c@{}}
\begin{tikzpicture}[>=stealth,scale=.84,baseline=-.5ex,
  every node/.style={font=\scriptsize,inner sep=1pt}]
  \coordinate (v4) at (-1.70,0);
  \coordinate (v1) at (0,1.10);
  \coordinate (v2) at (-.45,0);
  \coordinate (v3) at (0,-1.10);
  \coordinate (v5) at (1.70,0);
  \draw[->,shorten >=4pt,shorten <=4pt] (v1)--(v2);
  \draw[->,shorten >=4pt,shorten <=4pt] (v2)--(v3);
  \draw[->,shorten >=4pt,shorten <=4pt] (v1)--(v4);
  \draw[->,shorten >=4pt,shorten <=4pt] (v2)--(v4);
  \draw[->,shorten >=4pt,shorten <=4pt] (v3)--(v4);
  \draw[->,shorten >=4pt,shorten <=4pt] (v5)--(v1);
  \draw[->,shorten >=4pt,shorten <=4pt] (v5)--(v2);
  \draw[->,shorten >=4pt,shorten <=4pt] (v5)--(v3);
  \draw[->,shorten >=4pt,shorten <=4pt] (v5) to[bend left=28]
    (v4);
  \draw[->,blue,shorten >=4pt,shorten <=4pt] (v3) to[bend right=48]
    (v1);
  \node[fill=none] at (v1) {$1$};
  \node[fill=none] at (v2) {$2$};
  \node[fill=none] at (v3) {$3$};
  \node[fill=none] at (v4) {$4$};
  \node[fill=none] at (v5) {$5$};
\end{tikzpicture}
& $\mathrel{\overset{\mu_5}{\rule{0.8cm}{0.4pt}}}$ &
\begin{tikzpicture}[>=stealth,scale=.84,baseline=-.5ex,
  every node/.style={font=\scriptsize,inner sep=1pt}]
  \coordinate (v4) at (-1.70,0);
  \coordinate (v1) at (0,1.10);
  \coordinate (v2) at (-.45,0);
  \coordinate (v3) at (0,-1.10);
  \coordinate (v5) at (1.70,0);
  \draw[->,shorten >=4pt,shorten <=4pt] (v1)--(v2);
  \draw[->,shorten >=4pt,shorten <=4pt] (v2)--(v3);
  \draw[->,shorten >=4pt,shorten <=4pt] (v1)--(v4);
  \draw[->,shorten >=4pt,shorten <=4pt] (v2)--(v4);
  \draw[->,shorten >=4pt,shorten <=4pt] (v3)--(v4);
  \draw[->,shorten >=4pt,shorten <=4pt] (v1)--(v5);
  \draw[->,shorten >=4pt,shorten <=4pt] (v2)--(v5);
  \draw[->,shorten >=4pt,shorten <=4pt] (v3)--(v5);
  \draw[->,shorten >=4pt,shorten <=4pt] (v4) to[bend right=28]
    (v5);
  \draw[->,blue,shorten >=4pt,shorten <=4pt] (v3) to[bend right=48]
    (v1);
  \node[fill=none] at (v1) {$1$};
  \node[fill=none] at (v2) {$2$};
  \node[fill=none] at (v3) {$3$};
  \node[fill=none] at (v4) {$4$};
  \node[fill=none] at (v5) {$5$};
\end{tikzpicture}
& $\mathrel{\overset{\mu_4}{\rule{0.8cm}{0.4pt}}}$ &
\begin{tikzpicture}[>=stealth,scale=.84,baseline=-.5ex,
  every node/.style={font=\scriptsize,inner sep=1pt}]
  \coordinate (v4) at (-1.70,0);
  \coordinate (v1) at (0,1.10);
  \coordinate (v2) at (-.45,0);
  \coordinate (v3) at (0,-1.10);
  \coordinate (v5) at (1.70,0);
  \draw[->,shorten >=4pt,shorten <=4pt] (v1)--(v2);
  \draw[->,shorten >=4pt,shorten <=4pt] (v2)--(v3);
  \draw[->,shorten >=4pt,shorten <=4pt] (v4)--(v1);
  \draw[->,shorten >=4pt,shorten <=4pt] (v4)--(v2);
  \draw[->,shorten >=4pt,shorten <=4pt] (v4)--(v3);
  \draw[->,shorten >=4pt,shorten <=4pt] (v5) to[bend left=28]
    (v4);
  \draw[->,blue,shorten >=4pt,shorten <=4pt] (v3) to[bend right=48]
    (v1);
  \draw[->,blue,shorten >=4pt,shorten <=4pt] (v1)--(v5);
  \draw[->,blue,shorten >=4pt,shorten <=4pt] (v2)--(v5);
  \draw[->,blue,shorten >=4pt,shorten <=4pt] (v3)--(v5);
  \node[fill=none] at (v1) {$1$};
  \node[fill=none] at (v2) {$2$};
  \node[fill=none] at (v3) {$3$};
  \node[fill=none] at (v4) {$4$};
  \node[fill=none] at (v5) {$5$};
\end{tikzpicture}
\\[-1pt]
$\SSclass$-status
&& $\Sclass$-status
&& $\Rclass$-status
\end{tabular}
\end{center}
\end{example}

We first describe the cyclic subquivers and their descents relative to a
chosen closed nucleus.

\begin{proposition}\label{prop:cyclic-subquivers}
Let $Q\in\Rclass\cup\Sclass\cup\SSclass$ with a closed nucleus decomposition
$N=D\sqcup\{e\}\sqcup U$.
\begin{enumerate}[label=(\roman*),leftmargin=2.2em,itemsep=2pt]
\item If $Q\in\Rclass$ and $Q|_{\{a,b,c\}}$ is cyclic, then
$Q|_{\{a,b,c\}}$ has descent $e$ if $\{a,b,c\}\subseteq N$, and has descent
$r$ otherwise.
\item If $Q\in\Sclass\cup\SSclass$ and $Q|_{\{a,b,c\}}$ is cyclic, then
$\{a,b,c\}\subseteq N$ and $Q|_{\{a,b,c\}}$ has descent $e$.
\end{enumerate}
\end{proposition}

\begin{proof}
Suppose that $\{a,b,c\}\nsubseteq N$. Without loss of generality, assume
$a\notin N$. If $e\notin\{a,b,c\}$, then $Q|_{\{a,b,c\}}$ is a full
subquiver of $Q|_{\{e\}^C}$. In the $\Rclass$-case,
$Q|_{\{a,b,c\}}$ has descent $r$, since $Q|_{\{e\}^C}$ has a global descent at $r$.
Similarly, if $\{a,b,c\}\cap D=\varnothing$ or
$\{a,b,c\}\cap U=\varnothing$, then $Q|_{\{a,b,c\}}$ has descent $r$ by considering
$Q|_{D^C}$ or $Q|_{U^C}$, respectively. Since $a\notin N$, at least one of
these three cases occurs. This proves (i). In the $\Sclass$- and
$\SSclass$-cases, $Q|_{\{a,b,c\}}$ is contained in one of the acyclic quivers
$Q|_{\{e\}^C}$, $Q|_{D^C}$, and $Q|_{U^C}$, contradicting its cyclicity.
Thus (ii) follows.
\end{proof}

\begin{theorem}\label{thm:nucleus-uniqueness}
Let $Q\in\Rclass\cup\Sclass\cup\SSclass$ with closed nucleus $Q|_N$. Then $Q$ has a unique closed nucleus.
If $Q\in\Rclass$, then its return is also unique.
\end{theorem}

\begin{proof}
\emph{The $\Sclass$- and $\SSclass$-cases.} By
\cref{prop:cyclic-subquivers}, every cyclic subquiver of $Q$ has descent $e$.
Hence the closed nucleus is unique.

\emph{The $\Rclass$-case.} Let $r$ be the return.
If a closed nucleus $Q|_{N'}$ distinct from $Q|_N$ exists, then
\cref{prop:cyclic-subquivers} shows that $Q|_{N'}$ has global descent at $r$.
By \cref{def:classes}, $Q|_{\{e\}^C}$ is a nucleus with global descent at $r$.
For any $v\in\{e\}^C\setminus\{r\}$, $\exists\,u\in\{e\}^C$ such that $Q|_{\{v,r,u\}}$ is cyclic with descent $r$.
Since $Q|_{N'}$ is closed and $r\in N'$, it follows that $\{e\}^C\subseteq N'$.
Similarly, $Q|_{D^C}$ is a nucleus with global descent at $r$, then $e\in N'$. Thus $V(Q)=\{e\}^C\cup\{e\}\subseteq N'$.
Since $N'\subseteq V(Q)$, we have $N'=V(Q)$. But for any $d\in D$ and $u\in U$,
$Q|_{\{d,e,u\}}$ is cyclic with descent $e$, contradicting that $Q|_{N'}$
has global descent at $r$. Thus the closed nucleus is unique.

By \cref{prop:cyclic-subquivers}, the return is the only other descent occurring
in a cyclic subquiver. Hence the return is unique.
\end{proof}

For $Q\in\Rclass\cup\Sclass\cup\SSclass$, we write $N=D\sqcup\{e\}\sqcup U$
for its closed nucleus decomposition and $r$ for its return in the
$\Rclass$-case.

We next describe the arrows between the closed nucleus and the remaining
vertices.

\begin{proposition}\label{prop:npm-decomposition}
Let $Q\in\Rclass\cup\Sclass\cup\SSclass$. Define
\[
 \nminus:=\{v\in N^C:v\longrightarrow N\},\qquad
 \nplus:=\{v\in N^C:N\longrightarrow v\}.
\]
Then
\[
 V(Q)=N\sqcup\nminus\sqcup\nplus.
\]
\end{proposition}

\begin{proof}
Fix $v\in N^C$. Without loss of generality, assume that $e\to v$.
We show that $N\to v$.
Fix $d\in D$ and $u\in U$.
If $v\to d$, then $Q|_{\{d,e,v\}}\subseteq Q|_{U^C}$ is cyclic, so
\cref{prop:cyclic-subquivers} forces $Q\in\Rclass$ and $v=r$. If $v\to u$
and $v\not\to d$, then $d\to v$ and
$Q|_{\{d,u,v\}}\subseteq Q|_{\{e\}^C}$ is cyclic, so
\cref{prop:cyclic-subquivers} forces $Q\in\Rclass$ and $v=r$. Thus $N\to v$
unless $Q\in\Rclass$ and $v=r$.

It remains to consider $Q\in\Rclass$ and $v=r$. By
\cref{lem:global-descent-structure}, the path $e\to r\to u$ in $Q|_{D^C}$
would force $u\to e$, contrary to $e\to u$, so $u\to r$; the path
$u\to r\to d$ in $Q|_{\{e\}^C}$ would force $d\to u$, contrary to
$u\to d$, so $d\to r$; see \cref{fig:npm-return}.

\enlargethispage{1.5\baselineskip}
\begin{figure}[H]
\centering
\setlength{\tabcolsep}{15pt}
\begin{tabular}{cc}
\begin{tikzpicture}[>=stealth,scale=.84,
  every node/.style={font=\small,inner sep=.8pt}]
  \coordinate (d) at (-1.15,0.78);
  \coordinate (e) at (0,1.48);
  \coordinate (u) at (1.15,0.78);
  \coordinate (r) at (0,-0.28);

  \draw[->,shorten >=4pt,shorten <=4pt] (d)--(e);
  \draw[->,shorten >=4pt,shorten <=4pt] (e)--(u);
  \draw[->,shorten >=4pt,shorten <=4pt] (u)--(d);
  \draw[->,shorten >=4pt,shorten <=4pt] (e)--(r);
  \draw[->,red,thick,dashed,shorten >=4pt,shorten <=4pt] (r)--(u);
  \draw[dashed,shorten >=4pt,shorten <=4pt] (d)--(r);

  \node[fill=none] at (d) {$d$};
  \node[fill=none] at (e) {$e$};
  \node[fill=none] at (u) {$u$};
  \node[fill=none] at (r) {$r$};
\end{tikzpicture}
&
\begin{tikzpicture}[>=stealth,scale=.84,
  every node/.style={font=\small,inner sep=.8pt}]
  \coordinate (d) at (-1.15,0.78);
  \coordinate (e) at (0,1.48);
  \coordinate (u) at (1.15,0.78);
  \coordinate (r) at (0,-0.28);

  \draw[->,shorten >=4pt,shorten <=4pt] (d)--(e);
  \draw[->,shorten >=4pt,shorten <=4pt] (e)--(u);
  \draw[->,shorten >=4pt,shorten <=4pt] (u)--(d);
  \draw[->,shorten >=4pt,shorten <=4pt] (e)--(r);
  \draw[->,red,thick,dashed,shorten >=4pt,shorten <=4pt] (u)--(r);
  \draw[->,red,thick,dashed,shorten >=4pt,shorten <=4pt] (r)--(d);

  \node[fill=none] at (d) {$d$};
  \node[fill=none] at (e) {$e$};
  \node[fill=none] at (u) {$u$};
  \node[fill=none] at (r) {$r$};
\end{tikzpicture}
\\[-1mm]
{\small (a) $r\to u$.}
& {\small (b) $r\to d$.}
\end{tabular}
\caption{}
\label{fig:npm-return}
\end{figure}

Since $d\in D$ and $u\in U$ were arbitrary, $N\to r$, as desired.
\end{proof}

\begin{proposition}\label{prop:npm-acyclicity}
Let $Q\in\Rclass\cup\Sclass\cup\SSclass$. Then the full
subquivers $Q|_{\nminus}$ and $Q|_{\nplus}$ are acyclic.
\end{proposition}

\begin{proof}
It suffices to consider $Q|_{\nminus}$. If $Q\in\Sclass\cup\SSclass$, the claim is
immediate from \cref{def:classes}. Now suppose that $Q\in\Rclass$ and
that $Q|_{\nminus}$ is not acyclic. Then it contains a cyclic subquiver,
which by \cref{prop:cyclic-subquivers} has vertex set $\{x,y,r\}$ for some
$x,y\in\nminus$. The full subquiver $Q|_{\{x,y,r,e\}}\subseteq Q|_{D^C}$
is a vortex with apex $e$, contradicting the vortex-freeness of the nucleus
$Q|_{D^C}$. Dually,
$Q|_{\nplus}$ is acyclic.
\end{proof}

\begin{lemma}\label{lem:npm-occupancy}
Let $Q\in\Rclass\cup\Sclass\cup\SSclass$.
\begin{enumerate}[label=\textup{(\roman*)},leftmargin=2.2em,itemsep=2pt]
\item If $Q\in\Rclass\cup\SSclass$, then $\nminus$ and $\nplus$ are nonempty.
\item If $Q\in\Sclass$ has a source, then $\nplus=\varnothing$. If it has a
sink, then $\nminus=\varnothing$.
\end{enumerate}
\end{lemma}

\begin{proof}
\emph{The $\SSclass$-case.} If $Q\in\SSclass$, its source belongs to $\nminus$
and its sink belongs to $\nplus$.

\emph{The $\Rclass$-case.} Without loss of generality, assume that the return
$r\in\nplus$. It remains only to check that $\nminus$ is nonempty.
If $\nminus=\varnothing$, then $Q|_{D^C}$ is acyclic,
contradicting that $Q|_{D^C}$ is a nucleus. Hence $\nminus$ is nonempty.

\emph{The $\Sclass$-case.} Without loss of generality, suppose that $Q$ has a
source $s$. Then $s\in\nminus$. If $\nplus\ne\varnothing$, the acyclicity of
$Q|_{D^C}$ implies $\nminus\to\nplus$. Since
$V(Q)=N\sqcup\nminus\sqcup\nplus$ and $Q|_{\nplus}$ is acyclic, $Q$ has a
sink, a contradiction. Thus $\nplus$ is empty.
\end{proof}

\begin{proposition}\label{prop:basic-class-mutations}
Let $Q\in\Rclass\cup\Sclass\cup\SSclass$. For every $v\in D\cup U$,
the pair $(Q,v)$ satisfies F1--F3, so $v$ is an exit in $Q$.
\end{proposition}

\begin{proof}
By \cref{prop:cyclic-subquivers}, every cyclic subquiver containing
$v\in D\cup U$ has descent $e$ or $r$, and hence $v$ is an ascent; thus F1
holds. Every $v\in D\cup U$ belongs to a cyclic subquiver, so F2 holds. For F3, without loss of
generality, assume that $v\in D$.
If $v$ is the apex of a vortex, vortex-freeness of $Q|_{\{e\}^C}$ and
$Q|_{U^C}$ forces the cyclic subquiver of that vortex to contain
$e$ and some $u\in U$,
contradicting $v\to e$ and $u\to v$.
Thus $(Q,v)$ satisfies F1--F3, so $v$ is an exit in $Q$ by
\cref{lem:f1-f3-exit}.
\end{proof}

\Needspace{12\baselineskip}
\section{Mutation properties}\label{sec:mutation-properties}

We study reduced mutation sequences starting in
$\Rclass_1\cup\Rclass_{2k}\cup\SSclass_{2k}\cup\Sclass_{2k}$ and derive
the exit properties needed in \cref{sec:cycle}.

\begin{lemma}\label{lem:mutation-to-r}
Let $Q$ be an abundant quiver with a nucleus $Q|_N$ and decomposition
$N=D\sqcup\{e\}\sqcup U$.
Let $v\in N^C$ satisfy $v\to N$ or $N\to v$. If
\[
 (Q|_{\{e\}^C},v),\qquad
 (Q|_{D^C},v),\qquad
 (Q|_{U^C},v)
\]
satisfy F1--F3, then $\mu_v(Q)\in\Rclass$ with closed nucleus
$\mu_v(Q)|_N=Q|_N$ and return $v$.
\end{lemma}

\begin{proof}
By \cref{lem:mutation-equivalent-nucleus}, the quivers
$\mu_v(Q)|_{\{e\}^C}$, $\mu_v(Q)|_{D^C}$, and $\mu_v(Q)|_{U^C}$
are each a nucleus of $\mu_v(Q)$ with global descent at $v$.
Every edge of $\mu_v(Q)$ is contained in one of these full subquivers, so $\mu_v(Q)$ is abundant.

Since $v\to N$ or $N\to v$, $\mu_v(Q)|_N=Q|_N$ remains a nucleus of
$\mu_v(Q)$.

We claim that every cyclic subquiver of $\mu_v(Q)$ with descent $e$ is
contained in $Q|_N$. Otherwise, there is a cyclic subquiver
$\mu_v(Q)|_{\{e,x,y\}}$ with descent $e$ and $\{e,x,y\}\nsubseteq N$.
It is contained in either $\mu_v(Q)|_{D^C}$ or $\mu_v(Q)|_{U^C}$,
each of which has a global descent at $v$.

Thus $Q|_N$ is closed in $\mu_v(Q)$, and $\mu_v(Q)\in\Rclass$ with
return $v$.
\end{proof}

We first apply \cref{lem:mutation-to-r} to reduced mutation sequences
starting in $\Rclass_1$.

\begin{theorem}\label{thm:r1-sequences}
Let $Q\in\Rclass_1$ with return $r$, and let
$\mathbf v=(v_1,\ldots,v_\ell)$ be a reduced mutation sequence with
$\ell\ge1$ and $v_1\ne r$. Then $\mu_{\mathbf v}(Q)\notin\Sclass\cup\SSclass$.
\end{theorem}

\begin{proof}
We proceed by induction on $\ell$, simultaneously for all $Q\in\Rclass_1$.

\emph{Case 1:} $v_1\in D\cup U$.
The pair $(Q,v_1)$ satisfies F1--F3 by
\cref{prop:basic-class-mutations}. By \cref{lem:mutation-equivalent-nucleus},
$\mu_{\mathbf v}(Q)$ is a nucleus and hence has neither a source nor a sink.

\emph{Case 2:} $v_1\in\nminus\cup\nplus$.
By the definition of the $\Rclass$-status, the quivers
$Q|_{\{e\}^C}$, $Q|_{D^C}$, and $Q|_{U^C}$ are each a nucleus with global
descent at $r$. Since $v_1\ne r$, \cref{lem:basic-f1-f3} shows that
$\bigl(Q|_{\{e\}^C},v_1\bigr)$, $\bigl(Q|_{D^C},v_1\bigr)$, and
$\bigl(Q|_{U^C},v_1\bigr)$ satisfy F1--F3. Hence
\cref{lem:mutation-to-r} gives $\mu_{v_1}(Q)\in\Rclass_1$ with return $v_1$.
For $\ell=1$, the conclusion follows. For $\ell>1$, reducedness gives
$v_2\ne v_1$, and the induction hypothesis applies to $(v_2,\ldots,v_\ell)$
with initial quiver $\mu_{v_1}(Q)$.

It remains to consider $v_1=e$. By definition, $\mu_e(Q|_N)$ is abundant and
acyclic. Since
$\nminus\to N\to\nplus$, the full subquivers
$Q|_{\{u,e,n^-\}}$ and $Q|_{\{e,d,n^+\}}$ are acyclic with elbow $e$,
including when $n^-=r$ or $n^+=r$. By \cref{prop:cyclic-subquivers},
$Q|_{\{e,n^-,n^+\}}$ is acyclic unless $r\in\{n^-,n^+\}$. In the acyclic
case, $e$ is its elbow; in the cyclic case, its descent is $r$, so $e$ is an
ascent. The rank-three mutation rule therefore gives an abundant quiver
$\mu_e(Q)$ whose cyclic subquivers are
precisely
\[
\mu_e(Q)|_{\{u,e,n^-\}},\qquad
\mu_e(Q)|_{\{e,d,n^+\}},\qquad
\mu_e(Q)|_{\{e,n^-,n^+\}},
\]
all with descent $e$, where $d\in D$, $u\in U$, $n^-\in\nminus$, and
$n^+\in\nplus$ are arbitrary; see \cref{fig:one-nucleus-transition}.

\begin{figure}[H]
\centering
\begin{tikzpicture}[scale=.91]
\fivecoords
\draw[fivearrow] (w)--(e);  \draw[fivearrow] (e)--(d);  \draw[fivearrow] (w)--(d);
\draw[fivearrow] (e)--(cm); \draw[fivearrow] (cm)--(w); \draw[fivearrow] (cm)--(d);
\draw[fivearrow] (cp)--(e); \draw[fivearrow] (w)--(cp); \draw[fivearrow,shorten >=9pt] (d)--(cp);
\draw[fivearrow,shorten >=10pt] (cm) to[bend right=14] (cp);
\node[fill=none,inner sep=1pt,font=\small] at (cm) {$n^-$};
\node[fill=none,inner sep=1pt,font=\small] at (e) {$e$};
\node[fill=none,inner sep=1pt,font=\small] at (w) {$u$};
\node[fill=none,inner sep=1pt,font=\small] at (cp) {$n^+$};
\node[fill=none,inner sep=1pt,font=\small] at (d) {$d$};
\end{tikzpicture}
\caption{The orientation of $\mu_e(Q)$ in the $\Rclass_1$-case.}
\label{fig:one-nucleus-transition}
\end{figure}

Every vertex of $\mu_e(Q)$ belongs to one of the displayed cyclic subquivers,
so $\mu_e(Q)$ has neither a source nor a sink. If $\ell=1$, the proof is
complete. Assume $\ell>1$. Then $v_2\ne e$ by reducedness.

\emph{Case 3:} $v_1=e$, $v_2\in D\cup U$.
Without loss of generality, take $v_2=d\in D$.
Since $Q|_{U\sqcup\{e\}\sqcup\nminus}$ is acyclic with orientation
$\nminus\to e\to U$ and $\nminus\to U$,
\cref{def:special-nucleus-types}\textup{(i)} shows that
$\mu_e(Q)|_{U\sqcup\{e\}\sqcup\nminus}$ is a $1$-nucleus with decomposition
$U\sqcup\{e\}\sqcup\nminus$.
Moreover,
\[
 U\sqcup\{e\}\sqcup\nminus\longrightarrow d
 \qquad\text{in }\mu_e(Q).
\]
For this nucleus and the vertex $d$, consider
$\bigl(\mu_e(Q)|_{\{e\}^C},d\bigr)$,
$\bigl(\mu_e(Q)|_{U^C},d\bigr)$, and
$\bigl(\mu_e(Q)|_{(\nminus)^C},d\bigr)$.
The cyclic subquivers listed above all have descent $e$, so $d$ is an ascent
in every cyclic subquiver containing it. The orientations in
\cref{fig:one-nucleus-transition} show that $d$ is neither a source nor a sink
in any of the three full subquivers. The only possible vortex with apex $d$
has cyclic subquiver
$\mu_e(Q)|_{\{u,e,n^-\}}$; each of the three full subquivers omits one of
$u,e,n^-$. These observations verify F1--F3 for each pair, and
\cref{lem:mutation-to-r} gives $\mu_{(e,d)}(Q)\in\Rclass_1$ with return $d$.
For $\ell=2$, the conclusion follows. For $\ell>2$, reducedness gives
$v_3\ne d$, and the induction hypothesis applies to $(v_3,\ldots,v_\ell)$
with initial quiver $\mu_{(e,d)}(Q)$.

\emph{Case 4:} $v_1=e$, $v_2\in\nminus\cup\nplus$.
Without loss of generality, assume $v_2\in\nminus$. Since every cyclic
subquiver of $\mu_e(Q)$ has descent $e$,
the vertex $v_2$ is an ascent in every cyclic subquiver containing it.
Moreover, \cref{fig:one-nucleus-transition} shows that $v_2$ is neither a
source nor a sink and is not the apex of a vortex. Hence
$(\mu_e(Q),v_2)$ satisfies F1--F3. By
\cref{lem:mutation-equivalent-nucleus}, $\mu_{\mathbf v}(Q)$ is a nucleus
and hence has neither a source nor a sink.
\end{proof}

We now analyze alternating mutations at $e$ and $w$ for quivers with a
$2k$-nucleus.

Let $\Qalt{0}\in\Rclass_{2k}\cup\SSclass_{2k}\cup\Sclass_{2k}$.
Without loss of generality, assume $U=\{w\}$. Define the alternating mutation sequence by
\begin{equation}\label{eq:alternating-sequence}
\Qalt{j}=\mu_{(i_1,\ldots,i_j)}(\Qalt{0}),\qquad
i_j=\begin{cases}
e,&j\text{ odd},\\
w,&j\text{ even},
\end{cases}
\quad 1\le j\le2k+2.
\end{equation}
Throughout this sequence, $N,D,U,\nminus$, and $\nplus$ refer to $\Qalt{0}$.

\Needspace{30\baselineskip}
\begin{theorem}\label{thm:alternating-mutations}
Let $\Qalt{0}\in\Rclass_{2k}\cup\SSclass_{2k}$, and assume $U=\{w\}$.
Define $\Qalt{j}$, $1\le j\le2k+2$, by
\eqref{eq:alternating-sequence}.

Every $\Qalt{j}$, $0\le j\le2k+2$, is abundant, and $\Qalt{j}$ has neither
a source nor a sink for $1\le j\le2k+2$. The following statements hold.
\begin{enumerate}[label=\textup{(\roman*)},leftmargin=2.35em,itemsep=1pt,topsep=1pt]
\item For $\Qalt{0}$:
\begin{enumerate}[label=\textup{(\alph*)},leftmargin=2.15em,itemsep=0pt,topsep=0pt]
\item For every $v\in D\sqcup\{w\}$, the pair $(\Qalt{0},v)$ satisfies
F1--F3.
\item For $v\in\nminus\sqcup\nplus$, assume that $v\ne r$ in the
$\Rclass_{2k}$-case and that $v$ is neither a source nor a sink in the
$\SSclass_{2k}$-case. Then
$\mu_v(\Qalt{0})\in\Rclass_{2k}$ with return $v$.
\end{enumerate}

\item For $1\le j<2k$, $(\Qalt{j},v)$ satisfies F1--F3 for every $v\notin\{e,w\}$.

\item For $j=2k$:
\begin{enumerate}[label=\textup{(\alph*)},leftmargin=2.15em,itemsep=0pt,topsep=0pt]
\item For $v\in D$, $\mu_v(\Qalt{2k})\in\Rclass_{2k}$ with return $v$.
\item For $v\in\nminus$, $\mu_v(\Qalt{2k})\in\Rclass_1$ with return $v$.
\item For $v\in\nplus$, $(\Qalt{2k},v)$ satisfies F1--F3.
\end{enumerate}

\item For $j=2k+1$:
\begin{enumerate}[label=\textup{(\alph*)},leftmargin=2.15em,itemsep=0pt,topsep=0pt]
\item For $v\in D$, $(\Qalt{2k+1},v)$ satisfies F1--F3.
\item For $v\in\nminus$, $\mu_v(\Qalt{2k+1})\in\Rclass_1$ with return $v$.
\item For $v\in\nplus$, $\mu_v(\Qalt{2k+1})\in\Rclass_1$ with return $v$.
\end{enumerate}

\item For $j=2k+2$:
\begin{enumerate}[label=\textup{(\alph*)},leftmargin=2.15em,itemsep=0pt,topsep=0pt]
\item For $v\in D$,
$\mu_v(\Qalt{2k+2})\in\Rclass_{2k}$ with return $v$.
\item For $v\in\nminus$, $(\Qalt{2k+2},v)$ satisfies F1--F3.
\item For $v\in\nplus$, $\mu_v(\Qalt{2k+2})\in\Rclass_1$ with
return $v$.
\item The pair $(\Qalt{2k+2},e)$ satisfies F1--F3.
\end{enumerate}
\end{enumerate}
\end{theorem}

The conclusions are summarized in \cref{fig:alternating-flowchart}.

\begingroup
\setlength{\intextsep}{6pt}
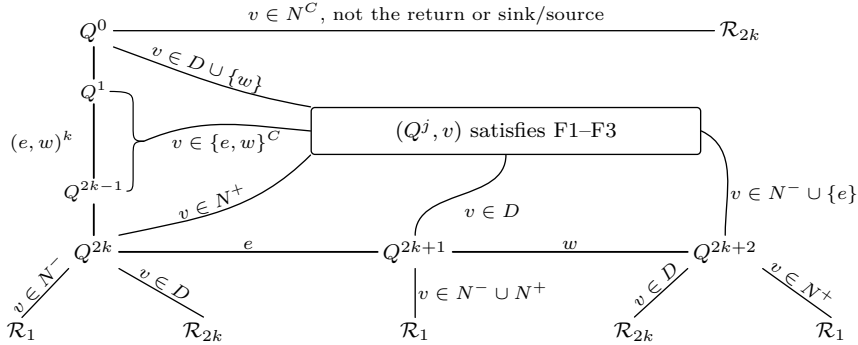
\begin{figure}[H]
\centering
\begin{tikzpicture}[
  x=1cm,y=.78cm,
  line cap=round,line join=round,
  every node/.style={font=\footnotesize,inner sep=1.4pt,fill=none},
  quiver/.style={font=\footnotesize,inner sep=1.8pt},
  result/.style={font=\footnotesize,inner sep=1.8pt},
  branch/.style={line width=.5pt},
  pathline/.style={line width=.7pt},
  labeltext/.style={font=\scriptsize,inner sep=.7pt},
  pathnode/.style={font=\scriptsize,inner sep=1pt,outer sep=0pt},
  condition/.style={draw,line width=.5pt,rounded corners=1.4pt,
    minimum width=5.15cm,minimum height=.62cm,font=\footnotesize}
]
  \node[quiver] (q0) at (0.8,3.75) {$Q^0$};
  \node[pathnode] (q1) at (0.8,2.72) {$Q^1$};
  \node[pathnode] (q2km1) at (0.8,1.03) {$Q^{2k-1}$};
  \node[quiver] (q2k) at (0.8,0) {$Q^{2k}$};
  \node[quiver] (q2k1) at (5.05,0) {$Q^{2k+1}$};
  \node[quiver] (q2k2) at (9.15,0) {$Q^{2k+2}$};
  \node[condition] (f) at (6.25,2.05) {$(Q^j,v)$ satisfies F1--F3};
  \node[result] (r0) at (9.35,3.75) {$\mathcal R_{2k}$};

  \draw[pathline] (q0.south) -- (q1.north);
  \draw[pathline] (q1.south) -- (q2km1.north);
  \draw[pathline] (q2km1.south) -- (q2k.north);
  \node[labeltext,anchor=east] at (0.57,1.86) {$(e,w)^k$};
  \draw[pathline] (q2k.east) -- node[labeltext,above] {$e$} (q2k1.west);
  \draw[pathline] (q2k1.east) -- node[labeltext,above] {$w$} (q2k2.west);

  \draw[branch] (q0.east) --
    node[labeltext,above=1.5pt,pos=.52] {$v\in N^C$, not the return or sink/source}
    (r0.west);
  \draw[branch] (q0.south east) to[out=-20,in=170]
    node[labeltext,above,sloped,pos=.46] {$v\in D\cup\{w\}$}
    (f.north west);

  \coordinate (braceout) at (1.46,1.875);
  \draw[line width=.45pt]
    (q1.east) .. controls (1.32,2.72) and (1.32,2.66) .. (1.32,2.46)
    -- (1.32,2.12)
    .. controls (1.32,1.93) and (1.32,1.875) .. (braceout)
    .. controls (1.32,1.875) and (1.32,1.82) .. (1.32,1.63)
    -- (1.32,1.29)
    .. controls (1.32,1.09) and (1.32,1.03) .. (q2km1.east);
  \draw[branch] (braceout) to[out=23,in=175]
    node[labeltext,below=1.5pt,pos=.51] {$v\in\{e,w\}^C$}
    (f.west);

  \draw[branch] (q2k.north east) to[out=12,in=-135]
    node[labeltext,above,sloped,pos=.46] {$v\in N^+$}
    (f.south west);
  \draw[branch] (q2k1.north) to[out=90,in=-90]
    node[labeltext,below right=2pt,pos=.47] {$v\in D$}
    (f.south);
  \draw[branch] (q2k2.north) to[out=90,in=-25]
    node[labeltext,right,pos=.34] {$v\in N^-\cup\{e\}$}
    (f.east);

  \node[result] (ra) at (-0.15,-1.35) {$\mathcal R_1$};
  \node[result] (rb) at (2.25,-1.35) {$\mathcal R_{2k}$};
  \node[result] (rc) at (5.05,-1.35) {$\mathcal R_1$};
  \node[result] (rd) at (7.95,-1.35) {$\mathcal R_{2k}$};
  \node[result] (re) at (10.55,-1.35) {$\mathcal R_1$};
  \draw[branch] (q2k.south west) --
    node[labeltext,above,sloped,pos=.48] {$v\in N^-$}
    (ra.north);
  \draw[branch] (q2k.south east) --
    node[labeltext,above,sloped,pos=.48] {$v\in D$}
    (rb.north);
  \draw[branch] (q2k1.south) --
    node[labeltext,right,pos=.46] {$v\in N^-\cup N^+$}
    (rc.north);
  \draw[branch] (q2k2.south west) --
    node[labeltext,above,sloped,pos=.48] {$v\in D$}
    (rd.north);
  \draw[branch] (q2k2.south east) --
    node[labeltext,above,sloped,pos=.48] {$v\in N^+$}
    (re.north);
\end{tikzpicture}
\caption{Mutation outcomes in \cref{thm:alternating-mutations}.
The $\Rclass$-labels give the status of $\mu_v(Q^j)$, with return $v$.
For $k=1$, the labels $Q^1$ and $Q^{2k-1}$ denote the same quiver.}
\label{fig:alternating-flowchart}
\end{figure}
\endgroup

We next prove \cref{thm:alternating-mutations}. By
\cref{prop:npm-decomposition},
$V(\Qalt{0})=\{e\}\sqcup\{w\}\sqcup D\sqcup\nminus\sqcup\nplus$.
Fix any $d\in D$. If $\nminus\ne\varnothing$, fix $n^-\in\nminus$; if
$\nplus\ne\varnothing$, fix $n^+\in\nplus$. Throughout the proof, write
\[
 b_{ij}=b_{ij}(\Qalt{0}),\qquad p_\ell=p_\ell(b_{ew}).
\]

\begin{lemma}\label{lem:alternating-ew-edges}
The full subquivers $\Qalt{j}|_D$, $\Qalt{j}|_{\nminus}$, and
$\Qalt{j}|_{\nplus}$ remain unchanged throughout
\eqref{eq:alternating-sequence}.
The orientations and weights of all edges joining $e$ or $w$ to
$D\cup\nminus\cup\nplus$ are determined by
\cref{lem:rank-three-alternating}.
\end{lemma}

\begin{proof}
The full subquivers $\Qalt{2k}|_{\{e,w,d\}}$,
$\Qalt{0}|_{\{e,w,n^-\}}$, and $\Qalt{0}|_{\{e,w,n^+\}}$ are acyclic by
\cref{def:special-nucleus-types}\textup{(ii)} and
\cref{prop:npm-decomposition}. The quivers
$\Qalt{2k}|_{\{e,w,d\}}$ and
$\Qalt{0}|_{\{e,w,n^-\}}$ have elbows $w$ and $e$, respectively. Hence
\cref{lem:rank-three-alternating} applies backward to
$\Qalt{j}|_{\{e,w,d\}}$ from $j=2k$ and forward to
$\Qalt{j}|_{\{e,w,n^-\}}$ from $j=0$.

In $\Qalt{0}|_{\{e,w,n^+\}}$, the vertex $e$ is a source, and in
$\Qalt{1}|_{\{e,w,n^+\}}$, the vertex $w$ is a source. The quiver
$\Qalt{2}|_{\{e,w,n^+\}}$ is acyclic with elbow $e$, so
\cref{lem:rank-three-alternating} applies from $j=2$. Finally, within the
full subquivers on $\{e,w,d\}$, the last two mutations merely reverse the
arrows incident to $e$ and $w$, since $e$ and $w$ are sources in
$\Qalt{2k}|_N$ and $\Qalt{2k+1}|_N$, respectively.

The quivers $\Qalt{0}|_D$, $\Qalt{0}|_{\nminus}$, and
$\Qalt{0}|_{\nplus}$ are acyclic by
\cref{rem:delete-global-descent,prop:npm-acyclicity}. The preceding orientations
show that, at every step, all arrows between the mutated vertex and each of
these three sets have the same orientation. Hence mutation does not change
any arrow joining two vertices of the same set.
\end{proof}

\begin{lemma}\label{lem:alternating-d-nminus}
For every $d\in D$, $n^-\in\nminus$, and $0\le j\le2k+2$, the quiver
$\Qalt{j}$ contains an arrow $n^-\to d$ of weight at least
$b_{n^-d}\ge2$.
\end{lemma}

\begin{proof}
Before each of the first $2k$ mutations, $n^-$ and $d$ lie on the same
side of the mutated vertex, by \cref{lem:alternating-ew-edges}; hence their
edge is unchanged. In $\Qalt{2k}$ we have $n^-\to e\to d$, so mutation at
$e$ increases the weight of $n^-\to d$. In $\Qalt{2k+1}$ we have
$n^-\to w\to d$, so mutation at $w$ increases it once more. Thus
$n^-\to d$ throughout the sequence, with weight at least $b_{n^-d}\ge2$.
\end{proof}

\begin{lemma}\label{lem:alternating-d-nplus}
For every $d\in D$, $n^+\in\nplus$, and $0\le j\le2k+2$, the quiver
$\Qalt{j}$ contains an arrow $d\to n^+$ of weight at least
$b_{dn^+}\ge2$.
\end{lemma}

\begin{proof}
In $\Qalt{0}$, we have $d\to n^+$. Mutation at $e$ increases its weight through
$d\to e\to n^+$, and the following mutation at $w$ increases it through
$d\to w\to n^+$. For $2\le j<2k$, the vertices $d$ and $n^+$ lie on the
same side of the mutated vertex, by \cref{lem:alternating-ew-edges}, so their
edge is unchanged.

Applying the Chebyshev recurrence \eqref{eq:rank-three-recurrence} to the full
subquivers on $\{e,w,d\}$ and $\{e,w,n^+\}$, respectively, gives
\begin{weightcalc}
 b_{de}(\Qalt{2k})
  &=-(p_{2k}b_{ed}+p_{2k-1}b_{wd})<0,\\
 b_{en^+}(\Qalt{2k})
  &=-(p_{2k-2}b_{en^+}+p_{2k-3}b_{wn^+})<0.
 \end{weightcalc}
Thus in $\Qalt{2k}$ we have $n^+\to e\to d$, and mutation at $e$ gives
\begin{weightcalc}
 b_{dn^+}(\Qalt{2k+1})
 &=b_{dn^+}(\Qalt{2k})
   -b_{de}(\Qalt{2k})b_{en^+}(\Qalt{2k})\\
 &=b_{dn^+}-b_{ed}b_{en^+}
   -(p_1b_{ed}+p_0b_{wd})b_{wn^+}\\
 &\quad
   -(p_{2k}b_{ed}+p_{2k-1}b_{wd})
    (p_{2k-2}b_{en^+}+p_{2k-3}b_{wn^+})\\
 &=b_{dn^+}-(p_{2k-1}b_{ed}+p_{2k-2}b_{wd})
   (p_{2k-1}b_{en^+}+p_{2k-2}b_{wn^+}).
 \end{weightcalc}
Applying the Chebyshev recurrence \eqref{eq:rank-three-recurrence} to the full
subquivers on $\{e,w,d\}$ and $\{e,w,n^+\}$, respectively, gives
\begin{weightcalc}
 b_{dw}(\Qalt{2k+1})&=p_{2k-1}b_{ed}+p_{2k-2}b_{wd},\\
 b_{n^+w}(\Qalt{2k+1})&=p_{2k-1}b_{en^+}+p_{2k-2}b_{wn^+}.
\end{weightcalc}
\Needspace{9\baselineskip}
Thus in $\Qalt{2k+1}$ we have $n^+\to w\to d$, and
\begin{weightcalc}
 b_{dn^+}(\Qalt{2k+1})
 &=b_{dn^+}
  +\lvert b_{dw}(\Qalt{2k+1})\rvert
   \cdot\lvert b_{n^+w}(\Qalt{2k+1})\rvert\\
 &>b_{dn^+}\ge2,\\[2pt]
 b_{dn^+}(\Qalt{2k+2})
 &=b_{dn^+}(\Qalt{2k+1})
   -b_{dw}(\Qalt{2k+1})b_{wn^+}(\Qalt{2k+1})\\
 &=b_{dn^+}(\Qalt{2k+1})
   -\lvert b_{dw}(\Qalt{2k+1})\rvert
    \cdot\lvert b_{n^+w}(\Qalt{2k+1})\rvert\\
 &=b_{dn^+}\ge2.
\end{weightcalc}
These inequalities also show that $e$ is a descent of
$\Qalt{2k}|_{\{e,d,n^+\}}$ and $w$ is a descent of
$\Qalt{2k+1}|_{\{w,d,n^+\}}$.
\end{proof}

\begin{lemma}\label{lem:alternating-nminus-nplus}
For every $n^-\in\nminus$, $n^+\in\nplus$, and $1\le j\le2k+2$, the quiver
$\Qalt{j}$ contains the arrow $n^-\to n^+$ of weight at least $2$.
\end{lemma}

\begin{proof}
In $\Qalt{0}$, we have $n^-\to e\to n^+$.

\emph{Case 1:} $\Qalt{0}|_{\{e,n^-,n^+\}}$ is acyclic.
Then $e$ is its elbow.

\emph{Case 2:} $\Qalt{0}|_{\{e,n^-,n^+\}}$ is cyclic.
Then $\Qalt{0}\in\Rclass$, and this cyclic subquiver has descent
$r\in\{n^-,n^+\}$ by \cref{prop:cyclic-subquivers}; hence $e$ is an
ascent.

In either case,
mutation at $e$ gives an arrow $n^-\to n^+$ of weight at least $2$ in
$\Qalt{1}$. Mutation at $w$ increases its weight, and no subsequent mutation
changes this edge, by \cref{lem:rank-three-edge-invariance}.
\end{proof}

The resulting orientations are shown in \cref{fig:alternating-orientations}.

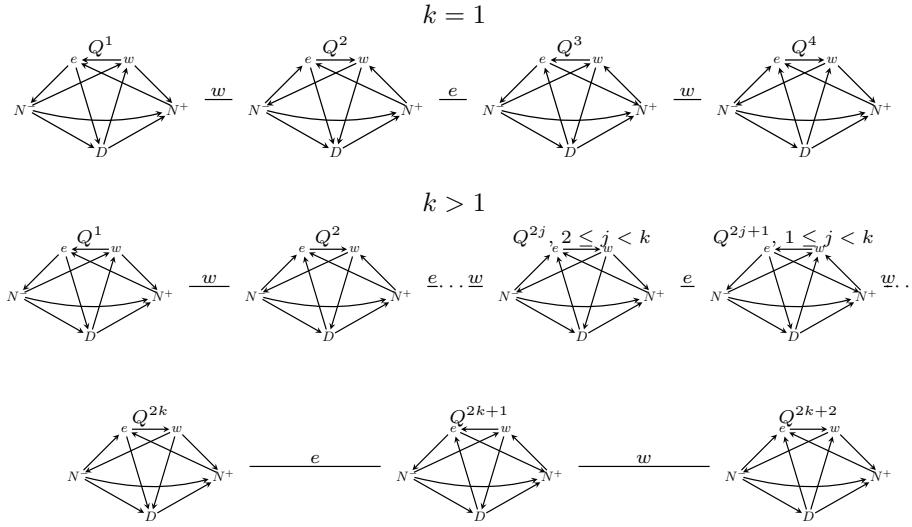
\begin{figure}[H]
\centering
\begin{minipage}{\textwidth}
\centering
\noindent $k=1$\par\smallskip
\begin{tikzpicture}[
  every node/.style={inner sep=0pt,align=center},
  segline/.style={line width=.45pt},
  lab/.style={font=\scriptsize,fill=none,inner sep=1pt}]
\node (q1one) at (-4.65,0) {\begin{tabular}{c}{\scriptsize $\Qalt{1}$}\\[-1mm]\scalebox{.60}{\GraphOne}\end{tabular}};
\node (q2one) at (-1.55,0) {\begin{tabular}{c}{\scriptsize $\Qalt{2}$}\\[-1mm]\scalebox{.60}{\GraphTwoAtOne}\end{tabular}};
\node (q3one) at (1.55,0) {\begin{tabular}{c}{\scriptsize $\Qalt{3}$}\\[-1mm]\scalebox{.60}{\GraphTwoKOne}\end{tabular}};
\node (q4one) at (4.65,0) {\begin{tabular}{c}{\scriptsize $\Qalt{4}$}\\[-1mm]\scalebox{.60}{\GraphTwoKTwo}\end{tabular}};
\draw[segline] (q1one.east) -- node[lab,above] {$w$} (q2one.west);
\draw[segline] (q2one.east) -- node[lab,above] {$e$} (q3one.west);
\draw[segline] (q3one.east) -- node[lab,above] {$w$} (q4one.west);
\end{tikzpicture}

\medskip
\noindent $k>1$\par\smallskip
\begin{tikzpicture}[
  every node/.style={inner sep=0pt,align=center},
  segline/.style={line width=.45pt},
  lab/.style={font=\scriptsize,fill=none,inner sep=1pt}]
\node (q1) at (-5.15,0) {\begin{tabular}{c}{\scriptsize $\Qalt{1}$}\\[-1mm]\scalebox{.57}{\GraphOne}\end{tabular}};
\node (q2) at (-2.00,0) {\begin{tabular}{c}{\scriptsize $\Qalt{2}$}\\[-1mm]\scalebox{.57}{\GraphTwo}\end{tabular}};
\node (leftdots) at (-.35,0) {{\scriptsize $\cdots$}};
\node (qeven) at (1.35,0) {\begin{tabular}{c}{\scriptsize $\Qalt{2j}$, $2\le j<k$}\\[-1mm]\scalebox{.57}{\GraphEven}\end{tabular}};
\node (qodd) at (4.15,0) {\begin{tabular}{c}{\scriptsize $\Qalt{2j+1}$, $1\le j<k$}\\[-1mm]\scalebox{.57}{\GraphOdd}\end{tabular}};
\node (rightdots) at (5.55,0) {{\scriptsize $\cdots$}};
\node (q2k) at (-4.35,-2.38) {\begin{tabular}{c}{\scriptsize $\Qalt{2k}$}\\[-1mm]\scalebox{.57}{\GraphTwoK}\end{tabular}};
\node (q2k1) at (0,-2.38) {\begin{tabular}{c}{\scriptsize $\Qalt{2k+1}$}\\[-1mm]\scalebox{.57}{\GraphTwoKOne}\end{tabular}};
\node (q2k2) at (4.35,-2.38) {\begin{tabular}{c}{\scriptsize $\Qalt{2k+2}$}\\[-1mm]\scalebox{.57}{\GraphTwoKTwo}\end{tabular}};
\draw[segline] (q1.east) -- node[lab,above] {$w$} (q2.west);
\draw[segline] (q2.east) -- node[lab,above] {$e$} (leftdots.west);
\draw[segline] (leftdots.east) -- node[lab,above] {$w$} (qeven.west);
\draw[segline] (qeven.east) -- node[lab,above] {$e$} (qodd.west);
\draw[segline] (qodd.east) -- node[lab,above] {$w$} (rightdots.west);
\draw[segline] (q2k.east) -- node[lab,above] {$e$} (q2k1.west);
\draw[segline] (q2k1.east) -- node[lab,above] {$w$} (q2k2.west);
\end{tikzpicture}
\end{minipage}
\caption{Orientations along the alternating mutation sequence
\eqref{eq:alternating-sequence}.}
\label{fig:alternating-orientations}
\end{figure}

\begin{lemma}\label{lem:alternating-interior}
For $1\le j<2k$, $(\Qalt{j},v)$ satisfies F1--F3 for every
$v\notin\{e,w\}$.
\end{lemma}

\begin{proof}
Fix $1\le j<2k$ and $v\notin\{e,w\}$. By
\cref{lem:alternating-ew-edges,lem:alternating-d-nminus,lem:alternating-d-nplus,lem:alternating-nminus-nplus},
the full subquivers on $D$, $\nminus$, and $\nplus$ are acyclic. For any two
distinct nonempty sets $A,B$ among
\[
 D,\qquad \{e\},\qquad \{w\},\qquad \nminus,\qquad \nplus
\]
either $A\to B$ or $B\to A$ in $\Qalt{j}$, with the direction shown in
\cref{fig:alternating-orientations}. Hence every cyclic subquiver contains at most one vertex
from each of $D$, $\nminus$, and $\nplus$. By \cref{lem:alternating-ew-edges,lem:alternating-d-nminus,lem:alternating-d-nplus,lem:alternating-nminus-nplus},
every cyclic subquiver containing $v$ has descent $e$ or $w$.
The orientations in \cref{fig:alternating-orientations} show that $v$ is
neither a source nor a sink. A full $4$-vertex subquiver containing two
vertices from the same set is either acyclic or contains two cyclic subquivers
obtained by exchanging these vertices. Hence it is not a vortex. Every
remaining full $4$-vertex subquiver has one of the forms
\[
 \Qalt{j}|_{\{e,w,n^-,n^+\}},\quad
 \Qalt{j}|_{\{w,d,n^-,n^+\}},\quad
 \Qalt{j}|_{\{e,d,n^-,n^+\}},\quad
 \Qalt{j}|_{\{e,w,d,n^+\}},\quad
 \Qalt{j}|_{\{e,w,d,n^-\}},
\]
whenever the indicated vertices exist. For $1\le j<2k$, the orientations in
\cref{fig:alternating-orientations} show that each listed full subquiver
contains either no cyclic subquiver or at least two cyclic subquivers. Hence
none is a vortex. Thus no $v\notin\{e,w\}$ is the apex of a vortex, and
$(\Qalt{j},v)$ satisfies F1--F3.
\end{proof}

\Needspace{6\baselineskip}
\begin{lemma}\label{lem:alternating-boundary}
The assertions in \cref{thm:alternating-mutations}\textup{(iii)--(v)} hold.
\end{lemma}

\begin{proof}
Assume $\Qalt{0}\in\Rclass_{2k}\cup\SSclass_{2k}$. By
\cref{lem:npm-occupancy}, both $\nminus$ and $\nplus$ are nonempty.
The orientations, cyclic subquivers, descents, and vortex configurations used
below follow from
\cref{lem:alternating-ew-edges,lem:alternating-d-nminus,lem:alternating-d-nplus,lem:alternating-nminus-nplus}.

\textup{(iii)} $j=2k$.

The full subquiver $\Qalt{2k}|_{\{e\}^C}$ is abundant and acyclic if $k>1$,
and is a nucleus with global descent at $w$ if $k=1$.

\textup{(a)} For $d\in D$, the quiver
$\Qalt{0}|_{\{e,w\}\sqcup\nminus}
=\mu_{(w,e)^k}(\Qalt{2k})|_{\{e,w\}\sqcup\nminus}$
is abundant and acyclic. Hence
$\Qalt{2k}|_{\{e,w\}\sqcup\nminus}$ is a $2k$-nucleus with global descent
at $w$. Moreover,
$\{e,w\}\sqcup\nminus\to d\to\nplus$ in $\Qalt{2k}$, so $d$ is a sink in
$\Qalt{2k}|_{\{e,w,d\}\sqcup\nminus}$. In $\Qalt{2k}|_{\{e\}^C}$,
the vertex $d$ is neither a source nor a sink. In
$\Qalt{2k}|_{\{w\}^C}$ and $\Qalt{2k}|_{(\nminus)^C}$, every cyclic
subquiver containing $d$ has descent $e$ or $w$. The only vortices among these
two full subquivers that contain $d$ have the form
$\Qalt{2k}|_{\{e,d,n^-,n^+\}}$, with apex $n^-$. These observations and \cref{lem:basic-f1-f3} show that
$(\Qalt{2k}|_{\{w\}^C},d)$, $(\Qalt{2k}|_{\{e\}^C},d)$, and
$(\Qalt{2k}|_{(\nminus)^C},d)$ satisfy F1--F3. Hence
\cref{lem:mutation-to-r} gives
$\mu_d(\Qalt{2k})\in\Rclass_{2k}$ with return $d$.

\textup{(b)} For $n^-\in\nminus$, the full subquiver
$\Qalt{2k}|_{\nplus\sqcup\{e\}\sqcup D}$ is a $1$-nucleus with global descent
at $e$. The vertex $n^-$ is neither a source nor a sink in
$\Qalt{2k}|_{\{e\}^C}$, and $\Qalt{2k}|_{D^C}$ is a nucleus
with global descent at $w$. In $\Qalt{2k}|_{(\nplus)^C}$, the only cyclic
subquiver containing $n^-$ is $\Qalt{2k}|_{\{e,w,n^-\}}$, whose descent is
$w$. The only vortices in this full subquiver containing $n^-$ have the form
$\Qalt{2k}|_{\{e,w,n^-,d\}}$, with apex $d$. These observations and \cref{lem:basic-f1-f3} show that
$(\Qalt{2k}|_{(\nplus)^C},n^-)$, $(\Qalt{2k}|_{\{e\}^C},n^-)$, and
$(\Qalt{2k}|_{D^C},n^-)$ satisfy F1--F3. Hence
\cref{lem:mutation-to-r} gives
$\mu_{n^-}(\Qalt{2k})\in\Rclass_1$ with return $n^-$.

\textup{(c)} For $n^+\in\nplus$, every cyclic subquiver containing
$n^+$ has descent $e$ or $w$. The only
vortices containing $n^+$ have the
form $\Qalt{2k}|_{\{e,d,n^-,n^+\}}$, with apex $n^-$. Hence
$(\Qalt{2k},n^+)$ satisfies F1--F3.

\textup{(iv)} $j=2k+1$.

\textup{(a)} For $d\in D$, the cyclic subquivers containing $d$
have the forms $\Qalt{2k+1}|_{\{w,d,n^+\}}$ and
$\Qalt{2k+1}|_{\{e,d,n^-\}}$, with descents $w$ and $e$, respectively. The
only vortices containing $d$ have the forms
$\Qalt{2k+1}|_{\{w,d,n^-,n^+\}}$ and
$\Qalt{2k+1}|_{\{e,d,n^-,n^+\}}$. A vortex of the first form has apex
$n^-$; a vortex of the second form has apex $n^+$. Hence
$(\Qalt{2k+1},d)$ satisfies F1--F3.

\textup{(b)} For $n^-\in\nminus$, the full subquiver
$\Qalt{2k+1}|_{\nplus\sqcup\{w\}\sqcup D}$ is a $1$-nucleus with global
descent at $w$. The full subquivers
$\Qalt{2k+1}|_{(\nplus)^C}$ and $\Qalt{2k+1}|_{D^C}$ are each a nucleus with
global descent at $e$. In $\Qalt{2k+1}|_{\{w\}^C}$, the only cyclic subquivers
containing $n^-$ have the form $\Qalt{2k+1}|_{\{e,d,n^-\}}$, with descent
$e$. The only vortices in this full subquiver containing $n^-$ have the form
$\Qalt{2k+1}|_{\{e,d,n^-,n^+\}}$, with apex $n^+$. These observations and \cref{lem:basic-f1-f3} show that
$(\Qalt{2k+1}|_{(\nplus)^C},n^-)$,
$(\Qalt{2k+1}|_{\{w\}^C},n^-)$, and
$(\Qalt{2k+1}|_{D^C},n^-)$ satisfy F1--F3. Hence
\cref{lem:mutation-to-r} gives
$\mu_{n^-}(\Qalt{2k+1})\in\Rclass_1$ with return $n^-$.

\textup{(c)} For $n^+\in\nplus$, the full subquiver
$\Qalt{2k+1}|_{D\sqcup\{e\}\sqcup\nminus}$ is a $1$-nucleus with global
descent at $e$. The full subquiver $\Qalt{2k+1}|_{D^C}$ is a nucleus with
global descent at $e$. In
$\Qalt{2k+1}|_{\{e\}^C}$ and $\Qalt{2k+1}|_{(\nminus)^C}$, every cyclic
subquiver containing $n^+$ has one of the forms
$\Qalt{2k+1}|_{\{w,d,n^+\}}$ and $\Qalt{2k+1}|_{\{e,w,n^+\}}$, with
descents $w$ and $e$, respectively.
The only vortices among these two full subquivers that contain $n^+$ have
the form $\Qalt{2k+1}|_{\{w,d,n^-,n^+\}}$, with apex $n^-$. These observations and \cref{lem:basic-f1-f3} show that
$(\Qalt{2k+1}|_{\{e\}^C},n^+)$, $(\Qalt{2k+1}|_{D^C},n^+)$, and
$(\Qalt{2k+1}|_{(\nminus)^C},n^+)$ satisfy F1--F3. Hence
\cref{lem:mutation-to-r} gives
$\mu_{n^+}(\Qalt{2k+1})\in\Rclass_1$ with return $n^+$.

\Needspace{8\baselineskip}
\textup{(v)} $j=2k+2$.

\textup{(a)} For $d\in D$, the quiver
$\Qalt{2}|_{\{e,w\}\sqcup\nplus}
=\mu_{(w,e)^k}(\Qalt{2k+2})|_{\{e,w\}\sqcup\nplus}$
is abundant and acyclic. Hence
$\Qalt{2k+2}|_{\{e,w\}\sqcup\nplus}$ is a $2k$-nucleus with global descent
at $w$. The full subquiver
$\Qalt{2k+2}|_{\{w\}^C}$ is abundant and acyclic, with $d$ neither a source
nor a sink, and $\Qalt{2k+2}|_{(\nplus)^C}$ is a nucleus with global descent
at $w$. In $\Qalt{2k+2}|_{\{e\}^C}$, the only cyclic subquivers containing
$d$ have the form $\Qalt{2k+2}|_{\{w,d,n^-\}}$, with descent $w$. The only
vortices in this full subquiver containing $d$ have the form
$\Qalt{2k+2}|_{\{w,d,n^-,n^+\}}$, with apex $n^+$. These observations and \cref{lem:basic-f1-f3} show that
$(\Qalt{2k+2}|_{\{w\}^C},d)$, $(\Qalt{2k+2}|_{\{e\}^C},d)$, and
$(\Qalt{2k+2}|_{(\nplus)^C},d)$ satisfy F1--F3. Hence
\cref{lem:mutation-to-r} gives
$\mu_d(\Qalt{2k+2})\in\Rclass_{2k}$ with return $d$.

\begingroup\emergencystretch=1em
\textup{(b)} For $n^-\in\nminus$, the cyclic subquivers containing
$n^-$ have the forms $\Qalt{2k+2}|_{\{e,w,n^-\}}$ and
$\Qalt{2k+2}|_{\{w,d,n^-\}}$, both with descent $w$. The only vortices
containing $n^-$ have the
form $\Qalt{2k+2}|_{\{w,d,n^-,n^+\}}$, with apex $n^+$. Hence
$(\Qalt{2k+2},n^-)$ satisfies F1--F3.
\par\endgroup

\textup{(c)} For $n^+\in\nplus$, the full subquiver
$\Qalt{2k+2}|_{D\sqcup\{w\}\sqcup\nminus}$ is a $1$-nucleus with global
descent at $w$. The full subquiver $\Qalt{2k+2}|_{\{w\}^C}$ is abundant and
acyclic, with $n^+$ neither a source nor a sink, and
$\Qalt{2k+2}|_{D^C}$ is a nucleus with global descent at $w$. In
$\Qalt{2k+2}|_{(\nminus)^C}$, the only cyclic subquivers containing $n^+$
have the form $\Qalt{2k+2}|_{\{e,w,n^+\}}$, with descent $w$. The only
vortices in this full subquiver containing $n^+$ have the form
$\Qalt{2k+2}|_{\{e,w,d,n^+\}}$, with apex $d$. These observations and \cref{lem:basic-f1-f3} show that
$(\Qalt{2k+2}|_{\{w\}^C},n^+)$, $(\Qalt{2k+2}|_{D^C},n^+)$, and
$(\Qalt{2k+2}|_{(\nminus)^C},n^+)$ satisfy F1--F3. Hence
\cref{lem:mutation-to-r} gives
$\mu_{n^+}(\Qalt{2k+2})\in\Rclass_1$ with return $n^+$.

\textup{(d)} Every cyclic subquiver containing $e$ has descent $w$;
its vertex set is $\{e,w,n^-\}$ or $\{e,w,n^+\}$. The only
vortices containing $e$ have the form $\Qalt{2k+2}|_{\{e,w,d,n^+\}}$, with
apex $d$. Hence $(\Qalt{2k+2},e)$ satisfies F1--F3.
\end{proof}

\begin{proof}[Proof of \cref{thm:alternating-mutations}]
The abundance assertion follows from
\cref{lem:alternating-ew-edges,lem:alternating-d-nminus,lem:alternating-d-nplus,lem:alternating-nminus-nplus}.
The assertion in \textup{(ii)} follows from \cref{lem:alternating-interior},
while \textup{(iii)--(v)} follow from \cref{lem:alternating-boundary}.

It remains to verify \textup{(i)}. For every $v\in D\sqcup\{w\}$, the pair
$(\Qalt{0},v)$ satisfies F1--F3 by \cref{prop:basic-class-mutations}, proving
\textup{(i)(a)}.

To prove \textup{(i)(b)}, let $v\in\nminus\sqcup\nplus$ satisfy the
restrictions in the statement.

\emph{Case 1:} $\Qalt{0}\in\Rclass_{2k}$.
The quivers $\Qalt{0}|_{\{e\}^C}$, $\Qalt{0}|_{D^C}$, and
$\Qalt{0}|_{U^C}$ are each a nucleus with global descent at $r\ne v$.

\emph{Case 2:} $\Qalt{0}\in\SSclass_{2k}$.
The quivers $\Qalt{0}|_{\{e\}^C}$, $\Qalt{0}|_{D^C}$, and
$\Qalt{0}|_{U^C}$ are acyclic, and \cref{prop:npm-decomposition} together
with the restriction on $v$ shows that $v$ is neither a source nor a sink
in any of them.

Thus
\cref{lem:basic-f1-f3} shows that $(\Qalt{0}|_{\{e\}^C},v)$,
$(\Qalt{0}|_{D^C},v)$, and $(\Qalt{0}|_{U^C},v)$ satisfy F1--F3 in either
case, and \cref{lem:mutation-to-r} gives
$\mu_v(\Qalt{0})\in\Rclass_{2k}$ with return $v$.

Finally, both $\nminus$ and $\nplus$ are nonempty by
\cref{lem:npm-occupancy}. The
orientations in \cref{fig:alternating-orientations} show that every vertex has
both an incoming and an outgoing arrow for $1\le j\le2k+2$. Hence $\Qalt{j}$
has neither a source nor a sink for $1\le j\le2k+2$.
\end{proof}

We next consider reduced mutation sequences starting in
$\Rclass_{2k}\cup\SSclass_{2k}$.

\begin{theorem}\label{thm:r2k-ss-sequences}
Let $Q\in\Rclass_{2k}\cup\SSclass_{2k}$, and let
$\mathbf v=(v_1,\ldots,v_\ell)$ be a reduced mutation sequence with
$\ell\ge1$. If $Q\in\Rclass_{2k}$ has return $r$, assume $v_1\ne r$; if
$Q\in\SSclass_{2k}$ has a source $s_1$ and a sink $s_2$, assume
$v_1\notin\{s_1,s_2\}$. Then $\mu_{\mathbf v}(Q)\notin\Sclass\cup\SSclass$.
\end{theorem}

\begin{proof}
We use strong induction on $\ell$, simultaneously for all
$Q\in\Rclass_{2k}\cup\SSclass_{2k}$. Set $\Qalt{0}=Q$ and, without loss of
generality, assume $U=\{w\}$. Retain the notation of
\cref{thm:alternating-mutations}.

\emph{Case 1:} $v_1\in D\cup\{w\}$.
The pair
$(Q,v_1)$ satisfies F1--F3 by
\cref{thm:alternating-mutations}\textup{(i)(a)}. By
\cref{lem:mutation-equivalent-nucleus}, $\mu_{\mathbf v}(Q)$ is a nucleus
and hence has neither a source nor a sink.

\emph{Case 2:} $v_1\in\nminus\cup\nplus$.
\Cref{thm:alternating-mutations}\textup{(i)(b)} gives
$\mu_{v_1}(Q)\in\Rclass_{2k}$ with return $v_1$. For $\ell=1$, the conclusion
follows. For $\ell>1$, reducedness gives $v_2\ne v_1$, and the induction
hypothesis applies to $(v_2,\ldots,v_\ell)$ with initial quiver
$\mu_{v_1}(Q)$.

Now let $v_1=e$, and let $j$ be maximal such that
$v_i\in\{e,w\}$ for $1\le i\le j$. By reducedness, these mutations alternate
between $e$ and $w$.

For $j>2k+2$, after the first $2k+2$ mutations the next mutation is at $e$.
By \cref{thm:alternating-mutations}\textup{(v)(d)},
$(\Qalt{2k+2},e)$ satisfies F1--F3. By
\cref{lem:mutation-equivalent-nucleus}, $\mu_{\mathbf v}(Q)$ is a nucleus
and hence has neither a source nor a sink. It remains to consider $j\le2k+2$.

\emph{Case 3:} $j=\ell\le2k+2$.
The quiver $\Qalt{j}$ has neither a source nor a sink by
\cref{thm:alternating-mutations}, which proves the conclusion.

\emph{Case 4:} $j<\ell$ and $j\le2k+2$.
The maximality of $j$ gives $v_{j+1}\notin\{e,w\}$. By
\cref{thm:alternating-mutations}\textup{(ii)--(iv) and (v)(a)--(c)}, either
$(\Qalt{j},v_{j+1})$ satisfies F1--F3, or
$\mu_{v_{j+1}}(\Qalt{j})$ belongs to $\Rclass_1$ or $\Rclass_{2k}$ with
return $v_{j+1}$.

If $(\Qalt{j},v_{j+1})$ satisfies F1--F3, then
\cref{lem:mutation-equivalent-nucleus} shows that $\mu_{\mathbf v}(Q)$ is a
nucleus and hence has neither a source nor a sink.

Suppose instead that
$\mu_{v_{j+1}}(\Qalt{j})\in\Rclass_1\cup\Rclass_{2k}$ with
return $v_{j+1}$. If $j+1=\ell$, then
$\mu_{\mathbf v}(Q)\notin\Sclass\cup\SSclass$. If $j+1<\ell$, reducedness
gives $v_{j+2}\ne v_{j+1}$; apply \cref{thm:r1-sequences} when
$\mu_{v_{j+1}}(\Qalt{j})\in\Rclass_1$ and the induction hypothesis when
$\mu_{v_{j+1}}(\Qalt{j})\in\Rclass_{2k}$.
\end{proof}

We finish with the $\Sclass_{2k}$-case.

\begin{theorem}\label{thm:s-sequences}
Let $Q\in\Sclass_{2k}$ with source or sink $s$, and let $e$ be
the global descent of its closed nucleus. Let
$\mathbf v=(v_1,\ldots,v_\ell)$ be a reduced mutation sequence with
$\ell\ge1$ and $v_1\notin\{e,s\}$. Then
$\mu_{\mathbf v}(Q)\notin\Sclass\cup\SSclass$.
\end{theorem}

\begin{proof}\leavevmode\par
\emph{Case 1:} $v_1\in D\cup U$.
Then $(Q,v_1)$ satisfies F1--F3 by
\cref{prop:basic-class-mutations}, so $\mu_{\mathbf v}(Q)$ has neither a source
nor a sink; hence the conclusion follows.

\emph{Case 2:} $v_1\in\nminus\cup\nplus$.
The quivers $Q|_{\{e\}^C}$, $Q|_{D^C}$, and
$Q|_{U^C}$ are acyclic, and, by \cref{prop:npm-decomposition} together with
$v_1\ne s$, $v_1$ is neither a source nor a sink in any of them. Thus
$\bigl(Q|_{\{e\}^C},v_1\bigr)$, $\bigl(Q|_{D^C},v_1\bigr)$, and
$\bigl(Q|_{U^C},v_1\bigr)$ satisfy F1--F3 by
\cref{lem:basic-f1-f3}, and
\cref{lem:mutation-to-r} gives
$\mu_{v_1}(Q)\in\Rclass_{2k}$ with return $v_1$.
If $\ell=1$, the conclusion follows. If $\ell>1$, reducedness gives
$v_2\ne v_1$, so \cref{thm:r2k-ss-sequences}, applied to
$(v_2,\ldots,v_\ell)$ with initial quiver $\mu_{v_1}(Q)$, gives the conclusion.
\end{proof}

The exit assertions now follow.

\begin{corollary}\label{cor:status-exits}
\begin{enumerate}[label=\textup{(\roman*)},leftmargin=2.4em,itemsep=4pt]
\item Let $Q\in\SSclass_{2k}$ with source $s_1$ and sink $s_2$.
Every $v\notin\{s_1,s_2\}$ is an exit in $Q$.
\item Let $Q\in\Sclass_{2k}$ with source or sink $s$, and let $e$ be the
global descent of its closed nucleus.
Every $v\notin\{e,s\}$ is an exit in $Q$.
\end{enumerate}
\end{corollary}

\begin{proof}
By \cref{thm:r2k-ss-sequences,thm:s-sequences}, any reduced mutation
sequence beginning at one of the indicated vertices ends outside
$\Sclass\cup\SSclass$ and hence cannot return to $Q$.
The conclusion follows from \cref{def:exit}.
\end{proof}

\Needspace{14\baselineskip}
\section{Nucleus-extension quivers and their unique cycles}\label{sec:cycle}

Let us show that every mutation away from the cycle is performed at an exit.

\begin{theorem}\label{thm:s-alternating-mutations}
Let $\Xj{0}\in\Sclass_{2k}$ with closed nucleus decomposition
$N=D\sqcup\{e\}\sqcup\{w\}$, where $w$ is the alternating vertex. Set $\Xj{j}=\mu_{(i_1,\ldots,i_j)}(\Xj{0})$ for
$1\le j\le2k+2$, with $i_j$ as in \eqref{eq:alternating-sequence}.
Here $\nminus$ and $\nplus$ are taken with respect to $\Xj{0}$.
\begin{enumerate}[label=\textup{(\roman*)},leftmargin=2.4em,itemsep=4pt]
\item If $\Xj{0}$ has a source, then $\Xj{2k}\in\Sclass_{2k}$ with alternating
vertex $e$ and closed nucleus decomposition \mbox{$\{e\}\sqcup\{w\}\sqcup\nminus$}.
For $1\le j<2k$, every $v\notin\{e,w\}$ is an exit in $\Xj{j}$.
\item If $\Xj{0}$ has a sink, then $\Xj{2k+2}\in\Sclass_{2k}$ with alternating
vertex $e$ and closed nucleus decomposition \mbox{$\{e\}\sqcup\{w\}\sqcup\nplus$}.
For $1\le j<2k+2$, every $v\notin\{e,w\}$ is an exit in $\Xj{j}$.
\end{enumerate}
\end{theorem}

\begin{proof}
By \cref{lem:npm-occupancy}, $\nplus=\varnothing$ in the source case and
$\nminus=\varnothing$ in the sink case. The computations in
\cref{lem:alternating-ew-edges,lem:alternating-d-nminus,lem:alternating-d-nplus}
apply with the orientations obtained from \cref{fig:alternating-orientations}
by omitting $\nplus$ and $\nminus$, respectively.

We first prove the exit assertions. Consider $1\le j<2k$ in the source case
and $1\le j<2k+2$ in the sink case. The preceding computations show that
$\Xj{j}$ is abundant and every cyclic subquiver of $\Xj{j}$ has descent
$e$ or $w$. As in the proof of \cref{lem:alternating-interior}, to verify
vortex-freeness it suffices to consider
\[
\begin{aligned}
 &\Xj{j}|_{\{e,w,d,n^-\}},\quad d\in D,\ n^-\in\nminus,
 &&\text{in the source case},\\
 &\Xj{j}|_{\{e,w,d,n^+\}},\quad d\in D,\ n^+\in\nplus,
 &&\text{in the sink case}.
\end{aligned}
\]
In the indicated ranges, each of these full subquivers contains at least two
cyclic subquivers. Thus none is a vortex, and every $v\notin\{e,w\}$ belongs
to a cyclic subquiver and is neither a source nor a sink. Hence
$(\Xj{j},v)$ satisfies F1--F3 for every $v\notin\{e,w\}$, so $v$ is an exit
by \cref{lem:f1-f3-exit}.

\begingroup\emergencystretch=1em
\textup{(i)} Suppose that $\Xj{0}$ has a source. By
\cref{lem:alternating-ew-edges,lem:alternating-d-nminus}, the quiver
$\Xj{0}|_{\{e,w\}\sqcup\nminus}
=\mu_{(w,e)^k}(\Xj{2k})|_{\{e,w\}\sqcup\nminus}$
is abundant and acyclic with the orientation required in
\cref{def:special-nucleus-types}\textup{(ii)}, so
$\Xj{2k}|_{\{e,w\}\sqcup\nminus}$ is a $2k$-nucleus with global descent
at $w$ and alternating vertex $e$. Moreover,
$\{e,w\}\sqcup\nminus\to D$ in $\Xj{2k}$, so this nucleus is closed.
The full subquivers $\Xj{2k}|_{\{e\}^C}$, $\Xj{2k}|_{\{w\}^C}$, and
$\Xj{2k}|_{(\nminus)^C}$ are acyclic. The sink of $\Xj{0}|_D$ is a sink of $\Xj{2k}$,
which has no source. Hence $\Xj{2k}\in\Sclass_{2k}$.
\par\endgroup

\textup{(ii)} Suppose that $\Xj{0}$ has a sink. By
\cref{lem:alternating-ew-edges,lem:alternating-d-nplus}, the quiver
$\Xj{2}|_{\{e,w\}\sqcup\nplus}
=\mu_{(w,e)^k}(\Xj{2k+2})|_{\{e,w\}\sqcup\nplus}$
is abundant and acyclic with the orientation required in
\cref{def:special-nucleus-types}\textup{(ii)}, so
$\Xj{2k+2}|_{\{e,w\}\sqcup\nplus}$ is a $2k$-nucleus with global descent
at $w$ and alternating vertex $e$. Moreover,
$D\to\{e,w\}\sqcup\nplus$ in $\Xj{2k+2}$, so this nucleus is closed.
The full subquivers $\Xj{2k+2}|_{\{e\}^C}$, $\Xj{2k+2}|_{\{w\}^C}$, and
$\Xj{2k+2}|_{(\nplus)^C}$ are acyclic. The source of $\Xj{0}|_D$ is a source of
$\Xj{2k+2}$, which has no sink. Hence $\Xj{2k+2}\in\Sclass_{2k}$.
\end{proof}

We now apply \cref{cor:status-exits} to the following mutation cycles.

\begin{definition}\label{def:nucleus-extension-quiver}
Let $n,m,k$ be integers satisfying
\begin{equation}\label{eq:parameters}
 n\geq4,\qquad 3\leq m\leq n-1,\qquad k\geq1.
\end{equation}
Fix integers $q_{ij}\geq2$ for $1\leq i<j\leq n$, and let $R$ be the
acyclic quiver on $[m]$ defined by
\begin{equation}\label{eq:reference-R}
 b_{ij}(R)=q_{ij}\qquad(1\leq i<j\leq m).
\end{equation}
Let $T=[m+1,n]$. The \emph{nucleus-extension quiver} $\Xj{0}$ is the
quiver on $[n]$ defined by
\begin{equation}\label{eq:nucleus-extension-quiver}
 \Xj{0}|_{[m]}=\mu_{(2,1)^k}(R),\qquad
 b_{ij}(\Xj{0})=q_{ij}\quad(1\leq i<j\leq n,\ j\in T).
\end{equation}
Consider the mutation sequence
\[
 (v_1,\ldots,v_{n+4k})
 :=(n,n-1,\ldots,m+1)(1,2)^k(m,m-1,\ldots,1)(2,1)^k.
\]
By \cite[Theorem~4.24]{EN26}, this sequence, starting from $\Xj{0}$, is a
simple mutation cycle; let $\Gamma$ denote the corresponding cycle in the
labeled mutation graph. Write
\begin{equation}\label{eq:cycle-quivers}
 \Xj{j}:=\mu_{(v_1,\ldots,v_j)}(\Xj{0})
 \qquad(1\leq j\leq n+4k).
\end{equation}
\end{definition}

Since $\mu_{(1,2)^k}\bigl(\Xj{0}|_{[m]}\bigr)=R$,
\cref{def:special-nucleus-types}\textup{(ii)} shows that
$\Xj{0}|_{[m]}$ is a $2k$-nucleus with decomposition
$[3,m]\sqcup\{1\}\sqcup\{2\}$ and alternating vertex $2$.
The defining orientation of
$\Xj{0}$ shows that $\Xj{0}|_{[m]}$ is closed and that the full subquivers
$\Xj{0}|_{\{1\}^C}$, $\Xj{0}|_{[3,m]^C}$, and $\Xj{0}|_{\{2\}^C}$ are
acyclic. Consequently, $\Xj{0}\in\Sclass_{2k}$.

For $m=n-1$, this construction recovers the Fomin--Neville family
\cite[Theorem~1.1]{FN25}.

\Needspace{8\baselineskip}
\begin{theorem}\label{thm:main}
$\Gamma$ is the unique simple cycle in the labeled mutation graph of the
mutation class of $\Xj{0}$.
\end{theorem}

\begin{proof}
By construction, $\Xj{0}$ and $\Xj{n-m}$ belong to $\Sclass_{2k}$, with a
sink $n$ and a source $m+1$, respectively. The source case of
\cref{thm:s-alternating-mutations}, applied to $\Xj{n-m}$ along $(1,2)^k$,
and the sink case, applied to $\Xj{0}$ along the reverse sequence
$(1,2)^{k+1}$, give
\[
 \Xj{n-m+2k},\qquad \Xj{n+2k-2}\in\Sclass_{2k},
\]
respectively. We verify the exit assertions in four ranges.

\emph{Case 1:} $0\le j\le n-m$.
The vertices $n,n-1,\ldots,m+1$ are mutated successively at sinks. For
$0<j<n-m$, $\Xj{j}\in\SSclass_{2k}$, and its source and sink label the two
incident edges of $\Gamma$. At $\Xj{0}$ these labels are the global descent
$1$ and sink $n$; at $\Xj{n-m}$ they are the global descent $1$ and source
$m+1$.

\emph{Case 2:} $n-m+1\le j\le n-m+2k-1$.
The source case applies to
$\Xj{n-m}$ along $(1,2)^k$, where $e=1$ and $w=2$ are precisely the labels
of the two edges of $\Gamma$ incident to $\Xj{j}$.

\emph{Case 3:} $n-m+2k\le j\le n+2k-2$.
The vertices $m,m-1,\ldots,3$ are mutated successively at sinks. For
$n-m+2k<j<n+2k-2$, $\Xj{j}\in\SSclass_{2k}$, and its source and sink label
the two incident edges of $\Gamma$. At $\Xj{n-m+2k}$ these labels are the
global descent $2$ and sink $m$; at $\Xj{n+2k-2}$ they are the global descent
$2$ and source $3$.

\emph{Case 4:} $n+2k-1\le j\le n+4k-1$.
The sink case applies to $\Xj{0}$ along
$(1,2)^{k+1}$, the reverse of the final mutation sequence, where $e=1$ and
$w=2$ are precisely the labels of the two edges of $\Gamma$ incident to
$\Xj{j}$.

Thus \cref{thm:s-alternating-mutations,cor:status-exits} show that every
mutation away from $\Gamma$ is performed at an exit. Hence $\Gamma$ is the
unique simple cycle in the labeled mutation graph.
\end{proof}

These four ranges are summarized in \cref{fig:cycle}.

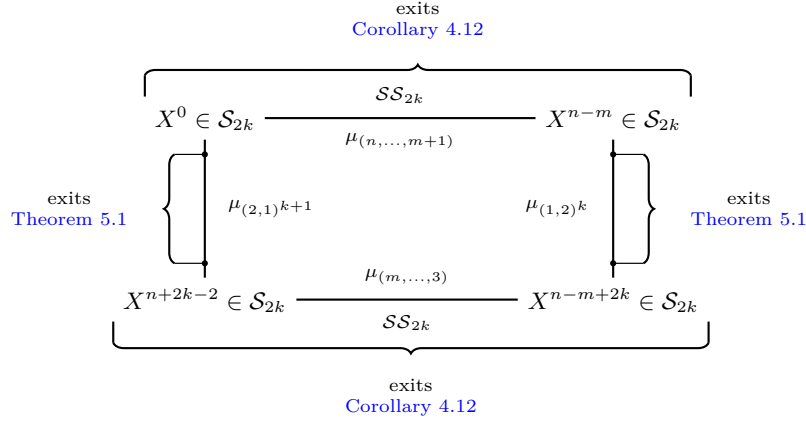
\begin{figure}[H]
\centering
\begin{tikzpicture}[thick,>=stealth,x=1cm,y=1cm]
\node[align=center,font=\small] (q0) at (-2.7,1.2)
 {$\Xj{0}\in\Sclass_{2k}$};
\node[align=center,font=\small] (q1) at (2.7,1.2)
 {$\Xj{n-m}\in\Sclass_{2k}$};
\node[align=center,font=\small] (q2) at (2.7,-1.2)
 {$\Xj{n-m+2k}\in\Sclass_{2k}$};
\node[align=center,font=\small] (q3) at (-2.7,-1.2)
 {$\Xj{n+2k-2}\in\Sclass_{2k}$};
\draw (q0) -- node[above=2pt] {$\scriptstyle\SSclass_{2k}$}
 node[below=2pt] {$\scriptstyle\mu_{(n,\ldots,m+1)}$} (q1);
\draw (q1) -- node[left=5pt] {$\scriptstyle\mu_{(1,2)^k}$} (q2);
\draw (q2) -- node[below=2pt] {$\scriptstyle\SSclass_{2k}$}
 node[above=2pt] {$\scriptstyle\mu_{(m,\ldots,3)}$} (q3);
\draw (q3) -- node[right=5pt] {$\scriptstyle\mu_{(2,1)^{k+1}}$} (q0);

\draw (q0.north west) -- ([yshift=8pt]q0.north west);
\draw (q1.north east) -- ([yshift=8pt]q1.north east);
\draw[decorate,decoration={brace,amplitude=4pt}]
 ([yshift=8pt]q0.north west) -- ([yshift=8pt]q1.north east)
 node[midway,above=10pt,align=center,font=\scriptsize]
 {exits\\\cref{cor:status-exits}};
\draw (q3.south west) -- ([yshift=-8pt]q3.south west);
\draw (q2.south east) -- ([yshift=-8pt]q2.south east);
\draw[decorate,decoration={brace,amplitude=4pt,mirror}]
 ([yshift=-8pt]q3.south west) -- ([yshift=-8pt]q2.south east)
 node[midway,below=10pt,align=center,font=\scriptsize]
 {exits\\\cref{cor:status-exits}};

\coordinate (ltop) at (-2.7,.72);
\coordinate (lbottom) at (-2.7,-.72);
\coordinate (rtop) at (2.7,.72);
\coordinate (rbottom) at (2.7,-.72);
\fill (ltop) circle (1.15pt) (lbottom) circle (1.15pt)
      (rtop) circle (1.15pt) (rbottom) circle (1.15pt);
\draw[line width=.45pt] (ltop) -- (-3.1,.72)
                         (lbottom) -- (-3.1,-.72);
\draw[decorate,decoration={brace,amplitude=4pt}]
 (-3.1,-.72) -- (-3.1,.72)
 node[midway,left=14pt,align=center,font=\scriptsize]
 {exits\\\cref{thm:s-alternating-mutations}};
\draw[line width=.45pt] (rtop) -- (3.1,.72)
                         (rbottom) -- (3.1,-.72);
\draw[decorate,decoration={brace,amplitude=4pt}]
 (3.1,.72) -- (3.1,-.72)
 node[midway,right=14pt,align=center,font=\scriptsize]
 {exits\\\cref{thm:s-alternating-mutations}};
\end{tikzpicture}
\caption{The cycle $\Gamma$.}
\label{fig:cycle}
\end{figure}

\Needspace{14\baselineskip}

\end{document}